\documentclass[11pt]{article}

\usepackage[margin=1in]{geometry}

\RequirePackage{tgtermes}
\RequirePackage{newtxtext}
\RequirePackage{newtxmath}
\RequirePackage{bm}
\RequirePackage{endnotes}

\usepackage{enumerate}
\usepackage{subcaption}
\usepackage{graphicx}
\usepackage{booktabs,tabularx,makecell}
\usepackage[table]{xcolor}
\usepackage{siunitx}
\usepackage{multirow}

\usepackage{algorithm}
\usepackage{algpseudocode}
\usepackage{tikz}

\usepackage{amsthm}

\usepackage{soul}
\usepackage{natbib}
 \bibpunct[, ]{(}{)}{,}{a}{}{,}%

\usepackage{pgfplots}
\pgfplotsset{compat=1.18}

\numberwithin{equation}{section}

\newtheorem{theorem}{Theorem}[section]
\newtheorem{lemma}[theorem]{Lemma}

\newtheorem{corollary}[theorem]{Corollary}
\newtheorem{prop}[theorem]{Proposition}

\newtheorem{definition}[theorem]{Definition}
\newtheorem{remark}[theorem]{Remark}
\newtheorem{assumption}[theorem]{Assumption}

\newtheorem{alg}{Algorithm}
\newenvironment{namedalg}[1]
  {\renewcommand\thealg{#1}\alg}
  {\endalg}

\def\itemrange#1{%
\addtocounter{enumi}{1}%
\edef\labelenumi{\theenumi--\noexpand\theenumi.}%
\addtocounter{enumi}{-1}%
\addtocounter{enumi}{#1}%
\item
\def\labelenumi{\theenumi.}}

\newcommand{\bb}{\mathbb}
\newcommand{\R}{\bb R}
\newcommand{\Z}{{\bb Z}}

\newcommand{\E}{\mathbb{E}}

\usepackage{appendix}
\newenvironment{APPENDICES}{%
  \appendix
  \section*{Appendix}
}{%
}

\DeclareMathOperator*{\argmax}{arg\,max}

\begin{document}

\title{Approximation Algorithms for Line Planning with Heterogeneous Fleets and Multiple
Resource Constraints}

\author{
  Hongyi Jiang\thanks{Department of Systems Engineering, City University of Hong Kong, \texttt{hongyi.jiang@cityu.edu.hk}}
  \and
  Igor Averbakh\thanks{Department of Management, University of Toronto Scarborough; Rotman School of Management, University of Toronto, \texttt{igor.averbakh@rotman.utoronto.ca}}
  \and
  Samitha Samaranayake\thanks{School of Civil and Environmental Engineering, Cornell University, \texttt{samitha@cornell.edu}}
}

\date{} 

\maketitle

\begin{abstract}
This paper studies line planning for urban bus networks that face multiple resource limits such as budget, labor, and emission caps while using heterogeneous fleets. The objective is to maximize total reward from serving passengers by assigning buses to candidate routes subject to capacity and resource constraints. The reward parameters are general and can encode diverse user preferences and multi-modal system configurations. Prior work typically assumes single resource constraints and homogeneous fleets, and often relies on methods that lack theoretical guarantees or computational tractability. We develop the first approximation algorithms with provable guarantees for this setting. For the cost-free variant, a randomized rounding scheme attains the optimal ratio $1-1/e$ which is tight unless $P = NP$. Leveraging this base algorithm, we derive extensions for the general case with arbitrary cost vectors, obtaining constant-factor approximation guarantees. To support large-scale application, we adapt the base algorithm to ensure computational scalability while preserving rigorous theoretical guarantees. Experiments on Greater Boston transit data demonstrate that our approach achieves 95\% to 98\% of the linear programming relaxation bound, whereas Gurobi solver fails on considerably smaller instances. Our experiments further show that heterogeneous fleets significantly outperform homogeneous ones and that multi-resource optimization is required to avoid significant resource limit violations, thereby underscoring the importance of our framework.
\end{abstract}



{\section{Introduction}

Urban transit systems face increasing demand driven by population growth, expanding cities, and stricter environmental regulations. Transit agencies must therefore optimize bus planning, a task that goes well beyond simply matching passenger demand with vehicle capacity. In practice, planners have to balance multiple constraints, including tight budgets, labor availability, and emission caps, while assigning a diverse fleet across competing routes. This creates a large-scale combinatorial optimization problem that must be solved within strict resource limits.
Addressing these challenges requires models that advance beyond traditional methods. While foundational research by \cite{newell1979some} and \cite{mandl1980evaluation} established important baselines for reducing costs and maximizing coverage, these early models were not built for environments with multiple constraints. Furthermore, despite the recent progress documented by \cite{duran2022survey} and \cite{schmidt2024planning}, current approaches still struggle to fully capture the complexities of modern transit networks.

First, the existing literature largely focuses on \textbf{single-resource constraints} and \textbf{homogeneous vehicle fleets}. Work on transit network design (\cite{szeto2014transit,buba2018differential,ahmed2019solving}) and research on line planning problem (\cite{perivier2021real,jiang2022approximation,jiang2024approximation}) exemplify this limitation. However, real-world operations involve complex trade-offs, such as balancing high-capacity electric buses (high purchase cost, low emissions) against conventional diesel buses (low purchase cost, high emissions). Modern systems therefore require multi-resource management for fleets with varying capacities, fuel types, and route restrictions (\cite{kocc2016thirty, schmidt2024planning}).

Second, current methodologies generally lack \textbf{theoretical performance guarantees}. Without such bounds on solution quality, transit planners have limited insight into how far their solutions deviate from the optimal, especially across diverse and complex scenarios where intuition cannot reliably gauge solution quality. While heuristic and metaheuristic methods are often used in practice to handle computational challenges \citep{iliopoulou2019metaheuristics,duran2022survey}, the absence of formal guarantees can be problematic for making informed planning decisions.

Finally, current frameworks suffer from critical limitations regarding \textbf{computational scalability}. Even moderately sized networks must evaluate hundreds of possible routes, coordinate diverse vehicle fleets, and serve tens of thousands of daily trips. Given this complexity, exact optimization becomes intractable (\cite{ibarra2015planning, goerigk2017line}), necessitating approximation algorithms that deliver high-quality solutions quickly.

\subsection{Summary of Contributions}

In response to these challenges, this paper develops a comprehensive algorithmic framework for transit line planning under multiple resource constraints that reflects real-world demand while offering rigorous and computationally tractable solutions. We consider a transportation system with a fleet of buses, potentially heterogeneous in capacity and operational characteristics, serving a collection of origin-destination (OD) pairs. Each bus operates on at most one line selected from a set of candidate routes where operation incurs costs across multiple resource types such as budget, labor, and emissions. The central challenge is to efficiently assign buses to lines while serving integral portions of passenger demands, strictly respecting bus capacities, resource limits, and demand constraints.

First, we introduce a \textbf{multi-resource heterogeneous fleet framework} equipped with a \textbf{rich and passenger-preference-aware setting} via our \emph{Line Planning with Resource Constraints} (LPRC) model. This formulation captures the complex decisions faced by transit agencies, including the simultaneous management of multiple resources and heterogeneous fleets. Moreover, the flexible utility framework accommodates user preferences such as bus type or road quality, as well as multi-modal system configurations involving first- and last-mile access via walking, or other auxiliary services. This generalizes existing models (\cite{perivier2021real,bertsimas2021data,jiang2022approximation,jiang2024approximation}), which are typically limited to single resources and homogeneous fleets, and is applicable to a broad class of assignment problems in transportation and logistics as discussed in Section \ref{sec:related-work}. Our numerical experiments in Section \ref{sec:computational-exp} quantify the value of this generalization. By effectively balancing distinct resource constraints, our heterogeneous-fleet approach yields an over 20\% higher objective value compared to optimal homogeneous solutions, even when homogeneous baselines enjoy superior routing flexibility. Furthermore, we demonstrate that simplifying the problem to a single constraint produces significantly infeasible plans with total budget overruns reaching up to 145\%.

Second, we provide algorithms with \textbf{rigorous theoretical guarantees} by building on the LPRC formulation to develop a suite of randomized approximation strategies. For the case with no line costs, we achieve the optimal approximation ratio of $1 - \frac{1}{e}$ (Section \ref{sec:no-line-cost}), which we prove cannot be improved unless $P = NP$. While this algorithm extends core ideas from \cite{fleischer2011tight}, our novel analysis addresses nontrivial complexities arising from divisible OD pair demands and variable reward values caused by candidate line assignment. The significance of this fundamental result is further underscored by two key corollaries presented in Section \ref{sec:extensions}. First, we generalize the framework to handle arbitrary costs, achieving an approximation ratio of $\frac{1}{2} - \frac{1}{2e} - \eta$ for any fixed $\eta > 0$. Second, we demonstrate that under slight resource augmentation, this guarantee improves to a near-optimal $1 - \frac{1}{e} - \eta$.

Finally, we implement \textbf{practical computational enhancements} on our algorithms. These include (i) a bus-aggregation technique (Proposition~\ref{prop:aggregate}) that reduces LP size without loss of optimality, (ii) sophisticated parameter tuning strategies (Proposition \ref{prop:alg-R-approx}) that automatically maximize the theoretical approximation guarantee for specific instances, and (iii) multi-stage improvements involving intelligent conflict resolution and residual capacity utilization. Extensive experiments using real-world transit datasets (\cite{mbta2014,bertsimas2021data}) show that our approach consistently attains 95\% to 98\% of the LP relaxation objective value on instances with up to 120 buses, 6 vehicle types, and nearly 94,000 passenger requests. This demonstrates superior scalability compared to commercial solvers like Gurobi, which encounter out-of-memory failures on problems with 90 or more buses. Extensive experiments on synthetic datasets are provided in Appendix~\ref{sec:app-synthetic}, which further demonstrate the consistency of our results across diverse problem instances.}


\subsection{Paper Organization}

The remainder of this paper is structured as follows. Section \ref{sec:related-work} reviews related work and positions our contributions within the existing literature on line planning and approximation algorithms. Section \ref{sec:formulation} presents our mathematical model and formulation. Sections \ref{sec:no-line-cost} through \ref{sec:extensions} develop our main theoretical results, progressing from simpler scenarios to the most general case. Section \ref{sec:practical} develops a practical algorithm with computational enhancements, which is then evaluated in Section \ref{sec:computational-exp}. Through comprehensive experiments in Section \ref{sec:computational-exp}, we first validate the necessity of the multi-resource, heterogeneous fleet model and then demonstrate our algorithm's near-optimal performance on large-scale instances that are intractable for state-of-the-art commercial solvers. We also perform an ex-post analysis to confirm that our assignments align with passengers' route preferences. Section \ref{sec:conclusion} summarizes our contributions and outlines future work.

\section{Related Work}\label{sec:related-work}

We review the most closely related lines of work, positioning our contributions within this existing literature.

{\textbf{Transit Line Planning. } 
Several recent surveys provide comprehensive overviews of models, algorithms, and research directions in transit line planning and transit network design. \cite{schmidt2024planning} offer an up-to-date and systematic review for transit line planning. \cite{duran2022survey} focus on developments in transit network design and frequency setting, while \cite{iliopoulou2019metaheuristics} survey metaheuristic approaches for the transit network design problem. 

From an algorithmic perspective, recent advances include data-driven optimization frameworks based on linear programming relaxations (\cite{bertsimas2021data}),  new complexity results and exact algorithms for non-pool-based line planning (\cite{heinrich2022algorithms,heinrich2023non}), and bi-objective metaheuristics optimizing heterogeneous fleets under environmental constraints (\cite{duran2020considering}). However, existing methods either lack theoretical performance guarantees or are limited to single-resource constraints. Specialized exact algorithms, such as Branch-and-Price, are rarely applied in transit network design and line planning, as they remain restricted to small, simplified instances or specific graph structures (\cite{duran2022survey,schmidt2024planning}). For instance, while \cite{borndorfer2007column} dynamically generate lines and solve the linear programming relaxation of the line-planning problem using column generation , they ultimately rely on a greedy heuristic to construct a feasible integer solution rather than employing a full Branch-and-Price algorithm. Furthermore, existing heuristics do not accommodate the multi-resource, heterogeneous fleet setting of LPRC. Consequently, we choose Gurobi as the baseline in our study.

Recent work has developed approximation algorithms with theoretical guarantees for transit line planning problems involving homogeneous bus fleets and single resource constraint.  \cite{perivier2021real} explore the Real-Time Line Planning Problem (RLPP), which can be considered as a special case of our problem. Specifically, in RLPP there is only a single resource constraint, the fleet consists of homogeneous buses that share the same candidate line set, and each origin-destination pair has a maximum demand of 1, which requires duplicating demands to handle practical problem sizes. To address RLPP, the authors employ a randomized rounding algorithm drawing inspiration from \cite{fleischer2011tight}, yielding a $(1-\frac{1}{e}-\epsilon)$-approximate solution with resource constraint violation probability bounded by $e^{-\frac{k}{3}\epsilon^2}$, where $k=\frac{1}{\max_{b}c_b}$ and $c_b$ denotes the cost for bus $b$.

Building on this work, \cite{jiang2022approximation} propose a randomized algorithm specifically for RLPP, attaining an approximation ratio of $\frac{1}{2}(1-\frac{1}{\sqrt{k}}-\epsilon)$. This work is further extended in \cite{jiang2024approximation}, where the algorithm is improved to achieve a constant approximation ratio of $\frac{1}{8}$.

While these algorithms offer theoretical guarantees, they are restricted to single resource constraints and homogeneous fleets. Real-world operations, however, involve multiple competing criteria and heterogeneous fleets (\cite{pternea2015sustainable, duran2020considering, ahern2022approximate}). Given the resulting combinatorial complexity, most approaches rely on heuristics to address these factors.

Our work substantially generalizes the theoretical approaches of \cite{perivier2021real,jiang2022approximation,jiang2024approximation} by incorporating multiple resource constraints, heterogeneous buses with distinct candidate line sets, and more efficient demand aggregation. These generalizations make the algorithms and analyses from previous work inapplicable to our setting, while our approximation ratios improve upon those achieved in these earlier papers.}

\textbf{Separable Assignment Problems and Theoretical Foundations.} The resource-unconstrained version of our problem can be considered as a generalization of the Separable Assignment Problem (SAP). In the standard SAP, there are $L$ bins and $P$ items. A reward $v_{lp}$ is generated when item $p$ is assigned to bin $l$. Additionally, each bin $l$ has its own packing constraint, which dictates that only certain subsets of items can be assigned to bin $l$. The goal is to optimize the total reward by finding an assignment of items to bins that satisfies all packing constraints and ensures that each item is allocated to at most one bin. \cite{fleischer2011tight} present a randomized $(1-\frac{1}{e})$-approximation algorithm for SAP, contingent upon the existence of a polynomial-time algorithm for the single-bin subproblem.

Our work extends this foundation but addresses two key differences that distinguish our line planning problem from SAP: (1) each ``item" (corresponding to passenger demand) can be divided into integral portions allocated to different ``bins" (buses), and (2) more importantly, the reward for assigning a unit of the same item to the same bin can vary (as it depends not only on the bus but also on the route), whereas in SAP it remains constant. To the best of our knowledge, no existing work addresses with theoretical guarantees the challenges posed by the combination of these two generalizations, even without resource constraints.

\textbf{Transportation-Specific Assignment Models.} Recent research has also explored applying assignment-type models to optimizing ride-sharing systems. \cite{gu2024algorithms} address the integration of public transit and ride-sharing of personal vehicles. They formulate the problem as an integer linear programming (ILP) model, which is a special case of SAP studied by \cite{fleischer2011tight}. \cite{luo2023efficient} investigate maximizing utility in ride-pooling assignments using a shareability graph representation. For their model it is not possible to develop a constant approximation ratio algorithm, unless $P=NP$, as it can be reduced to the multi-dimensional matching problem (\cite{hazan2006complexity}). These works illustrate the potential of assignment-type models (and, therefore, our study) for optimizing ride-sharing systems. For further literature in this area, readers can explore the references provided in \cite{luo2023efficient} and \cite{gu2024algorithms}.

\textbf{Assignment Problem Variants.} Our framework connects to the broader literature on assignment problems. SAP itself generalizes the Generalized Assignment Problem (GAP) (\cite{fisher1986multiplier}), where items with bin-specific sizes and values must be packed into capacity-constrained bins to maximize the total value. While related problems like the all-or-nothing generalized assignment problem (AGAP) (\cite{adany2016all,sarpatwar2018generalized}) and hypermatching assignment problem (HAP) (\cite{cygan2013sell}) have constant approximation algorithms, their different objectives and structures limit their applicability to our problem. GAP itself generalizes the multiple knapsack problem (\cite{hung1978algorithm}) by allowing bin-dependent item values and capacities.

\textbf{Submodular Optimization and its Inapplicability.} Another extensively studied area involves optimizing nonnegative submodular functions subject to knapsack constraints. Numerous algorithmic developments have been made in this direction (\cite{kulik2009maximizing,lee2009non, fadaei2011maximizing, chekuri2014submodular, feldman2011unified}). However, methods developed for maximizing submodular functions are not directly applicable to our problem because of the absence of the submodularity property (\cite{perivier2021real}). This holds true even when fractional assignment of demand is allowed.

\section{Problem Formulation}\label{sec:formulation}

To formalize the transit line planning problem, we consider a transportation network with a heterogeneous fleet of $M$ buses. Each bus $b$ can be assigned to a line chosen from a candidate line set $\mathcal{L}_b$, where each line is a directed path. The system serves a set $D$ of Origin-Destination (OD) pairs. Each pair $(i,j)$ has a positive integer demand $d_{(i,j)}$. A bus $b$ operating on line $l$ can carry an integer portion of this demand, denoted by $\xi_{b,l,(i,j)} \in [0,d_{(i,j)}]\cap \Z$, provided passengers of OD pair $(i,j)$ can access a boarding stop and reach their destination from an alighting stop on line $l$. This service yields a reward of $v_{b,l,(i,j)} \cdot \xi_{b,l,(i,j)}$, where $v_{b,l,(i,j)} \geq 0$ is a reward parameter that can be calibrated to reflect user preferences and system-level considerations (e.g., first- and last-mile access costs in a multi-modal network); see Section~\ref{sec:model-setting} for a detailed discussion. Let $D^{b,l} \subset D$ denote the set of OD pairs $(i,j)$ with $v_{b,l,(i,j)} > 0$, i.e., those that can be meaningfully served by bus $b$ on line $l$. For each bus $b$ and line $l \in \mathcal{L}_b$, every OD pair $(i,j) \in D^{b,l}$ is assumed to have a unique sub-path on line $l$ operated by bus $b$.

The system operates under $K$ types of resource constraints (e.g., financial budgets or emission caps). We normalize the available budget for each resource to 1. Consequently, assigning bus $b$ to operate on line $l$ incurs a normalized cost $c_{b,l}^{(k)} \in [0, 1]$ for each resource $k \in \{1,\dots,K\}$. Unused buses incur zero costs.

Let $D^{b,l}(e) \subseteq D^{b,l}$ be the set of OD pairs whose passengers traverse arc $e$ on bus $b$ along line $l$. Additionally, let $\mathcal{L}_b(i,j) \subseteq \mathcal{L}_b$ denote the subset of lines that can serve OD pair $(i,j)$ (i.e., $v_{b,l,(i,j)} > 0$).

The objective is to maximize total generated reward by simultaneously optimizing bus-to-line assignments and passenger demand allocation. We define the problem, \textbf{Line Planning with Resource Constraints (LPRC)}, using the following constraints:

\begin{enumerate}    
    \item Each bus is assigned to exactly one line. For clarity, we assume that each $\mathcal{L}_b$ contains a ``dummy" line with 0 costs and 0 rewards not passing through any OD pair; assigning a bus to that line is equivalent to not using the bus.
    \item The total usage of the resources incurred by line assignment  should not exceed the given limit for each resource type. Note that we scaled all resource limits to 1, with their costs scaled proportionally.
    \item For each OD pair $(i,j)$, the total demand served by all buses must not exceed $d_{(i,j)}$.
    \item Only an integer portion of the demand of any OD pair can be served by any bus.
    \item For a particular bus \(b\) and line \(l\in \mathcal{L}_b\), if bus $b$  operates on line $l$, the total demand carried by the bus at any point (in any arc) should not exceed the capacity of the bus $C_b\in \Z$, i.e., the demand allocation should reside in $\mathcal{P}(b,l)\cap \Z^{|D^{b,l}|}$, where $\mathcal{P}(b,l)$ is a polyhedron defined as follows:
\begin{align}\label{eq:P(b,l)}
    \mathcal{P}(b,l):=\Bigg\{\xi_{b,l}\in \prod_{(i,j)\in D^{b,l}}\left[0,d_{(i,j)}\right]:&\sum_{(i,j)\in D^{b,l}(e)}\xi_{b,l,(i,j)}\leq C_b,~~\forall e\in l \Bigg\}.
\end{align}



\end{enumerate}

LPRC is formulated as a linear integer programming problem that seeks to optimize both binary decisions (selection of lines for buses) and bounded-integer decisions (the demand portions to be served by specific buses) simultaneously. 
It can be formulated as follows:
\begin{subequations}\label{eq:ILP-original}
\begin{align}
    \max_{\lambda,\xi}~~~~~~~~~~&\sum_{b\in [M]}\sum_{l\in \mathcal{L}_b}\sum_{(i,j)\in D^{b,l}}v_{b,l,(i,j)}\xi_{b,l,(i,j)}\\
    s.t.~~~~~~~~~~~~& \sum_{l \in \mathcal{L}_b}\lambda_{b,l} = 1~~~~\forall b\in [M]\label{eq:ILP-one-line-each-bus}\\
    &\sum_{b\in [M]}\sum_{l\in \mathcal{L}_b} c_{b,l}^{(k)}\cdot \lambda_{b,l}\leq 1~~~~\forall k\in [K]\label{eq:ILP-bedget}\\
    &\sum_{b\in [M]}\sum_{l\in \mathcal{L}_b(i,j)}\xi_{b,l,(i,j)}\leq d_{(i,j)}~~~~\forall (i,j)\in D\label{eq:ILP-respect-demand}\\
    &\xi_{b,l,(i,j)}\leq \lambda_{b,l}\cdot d_{(i,j)} ~~~\forall b\in [M], l\in \mathcal{L}_b,(i,j)\in D^{b,l} \subseteq D\label{eq:ILP-serve-by-open-line}\\
    &\left(\xi_{b,l,(i,j)}\right)_{(i,j)\in D^{b,l}}\in \mathcal{P}(b,l)~~~\forall b\in [M], l\in \mathcal{L}_b\label{eq:ILP-capacity-constraints}\\
    &\lambda_{b,l}\in \{0,1\}~~~\forall b\in [M], l\in \mathcal{L}_b\\
    &\xi_{b,l,(i,j)}\in \Z~~~\forall b\in [M], l\in \mathcal{L}_b,(i,j)\in D^{b,l} \subseteq D
\end{align}
\end{subequations}

\noindent Here, the binary decision variable $\lambda_{b,l}$  indicates whether bus $b$ is assigned to line $l$ or not. If $\lambda_{b,l}=1$, bus $b$ is assigned to line $l$. If $\lambda_{b,l}=0$, bus $b$ is not assigned to line $l$. The integer decision variable $\xi_{b,l,(i,j)}$ represents the demand from an OD pair $(i,j)$ that is served by bus $b$ on line $l$. 

The objective function represents the total generated reward (or social welfare).
Constraint (\ref{eq:ILP-one-line-each-bus}) ensures that each bus $b$ is assigned to exactly one line from its candidate line set $\mathcal{L}_b$. Constraint (\ref{eq:ILP-bedget}) guarantees that the costs incurred by line assignments do not exceed the available resources.
Constraint (\ref{eq:ILP-respect-demand}) ensures that the total demand for each OD pair is not over-served.  Constraint (\ref{eq:ILP-serve-by-open-line}) ensures that we can only allocate a portion of the demand of OD pair $(i,j)$ to bus $b$ operating on line $l$ if bus $b$ runs on line $l$. Constraint (\ref{eq:ILP-capacity-constraints}) ensures that the total demand served by a bus on a line for all OD pairs does not exceed the bus capacity $C_b$.

\subsection{Discussion of Model Settings}\label{sec:model-setting}
Following \cite{goerigk2017line}'s notion, \emph{Line Planning with Route Assignment} optimizes centralized demand-to-service assignments (see also \cite{bertsimas2021data, schmidt2014integrating, borndorfer2007column, schobel2006line}), but can be criticized for routing passengers along undesirable paths. To address this, \cite{goerigk2017line} study \emph{Line Planning with Route Choice}, which adds constraints ensuring passengers only travel on shortest paths. Our formulation bridges these two paradigms through the reward parameters $v_{b,l,(i,j)}$. Setting $v_{b,l,(i,j)}=1$ recovers the Route Assignment model by solely maximizing ridership. Alternatively, calibrating these parameters to penalize detours and first- and last-mile access costs (see Section~\ref{sec:computational-exp}) aligns the model with user preferences, and sufficiently strict penalties effectively recover the Route Choice model. The same mechanism naturally extends to encoding other preference dimensions such as bus type or size, road quality, and multi-modal access configurations.

In addition, we employ a predefined line pool rather than generating lines dynamically, following a widely adopted modeling paradigm in the literature (e.g.,  \cite{hartleb2023modeling, friedrich2017integrating, schobel2012line}). This choice is further motivated by algorithmic considerations, as our goal is to derive a constant-factor approximation algorithm, and \cite{perivier2021real} proved that such approximation guarantees are unattainable when lines are generated dynamically.

Furthermore, our formulation does not account for transfers between bus lines. However, the general reward parameters $v_{b,l,(i,j)}$ are flexible enough to capture the costs of access and egress legs, including the distance from a passenger's origin to the boarding stop and from the alighting stop to their destination. Since passengers already incur implicit transfers during the first and last miles of their trip, imposing additional bus-to-bus transfers mid-journey would be excessive and further degrade the user experience. 


\section{LPRC without Line Costs}\label{sec:no-line-cost}

In this section, we assume that $c_{b,l}^{(k)}=0$ for all $b\in [M]$, $l\in \mathcal{L}_b$ and $k\in [K]$. Thus we can omit resource constraints (\ref{eq:ILP-bedget}) in  (\ref{eq:ILP-original}) and (\ref{eq:LP-budgets}) in the LP relaxation (\ref{LP-relaxation}) in this section. {For this scenario, we develop Algorithm \ref{alg:simple-alg}, which achieves the optimal approximation ratio of $1 - \frac{1}{e}$.}

\noindent{\bf An overview of finding and analyzing a solution for (\ref{eq:ILP-original}):} We obtain a high-quality integer solution by applying randomized rounding (Algorithm \ref{alg:simple-alg}) to the LP relaxation (Section \ref{sec:relaxation}). The approximation ratio is established by bounding an auxiliary Algorithm \ref{alg:benchmark-nc} and showing that Algorithm \ref{alg:simple-alg} outperforms this benchmark, thereby inheriting its performance guarantees.



\subsection{Linear Relaxation of (\ref{eq:ILP-original})}\label{sec:relaxation}



{In this section, we present a linear relaxation of (\ref{eq:ILP-original}) inspired by \cite{fleischer2011tight}, which is crucial for developing the subsequent randomized rounding approximation algorithms.}

Assume that the set of integer points in the polytope $\mathcal{P}(b,l)$ defined in (\ref{eq:P(b,l)}) has cardinality $n_{b,l}$. Let ${\theta}^{(t)}_{b,l} := \left({\theta}^{(t)}_{b,l,(i,j)}\right)_{(i,j) \in D^{b,l}}$ denote an integer point in $\mathcal{P}(b,l)$ for $t \in [n_{b,l}]$, {representing a feasible allocation of OD demands to bus $b$ on line $l$ under the capacity constraints. Using these integer points, we formulate the following LP relaxation of (\ref{eq:ILP-original}):}
\begin{subequations}\label{LP-relaxation}
\begin{align}
    \max_{\lambda}~~~~~~~~~~&\sum_{b\in [M]}\sum_{l\in \mathcal{L}_b}\sum_{t \in [n_{b,l}]}\sum_{(i,j)\in D^{b,l}}v_{b,l,(i,j)}{\theta}_{b,l,(i,j)}^{(t)}\lambda_{b,l}^{(t)}\\
    s.t.~~~~~~~~~~~~& \sum_{l\in\mathcal{L}_b}\sum_{t \in [n_{b,l}]}\lambda_{b,l}^{(t)} = 1~~~~\forall b\in [M]\label{eq:LP-one-line-each-bus}\\
    &\sum_{b\in [M]}\sum_{l\in \mathcal{L}_b}\sum_{t\in [n_{b,l}]} c_{b,l}^{(k)}\cdot \lambda_{b,l}^{(t)}\leq 1~~~~\forall k\in [K]\label{eq:LP-budgets}\\
    &\sum_{b\in [M]}\sum_{l\in \mathcal{L}_b(i,j)}\sum_{t \in [n_{b,l}]}\lambda_{b,l}^{(t)}{\theta}_{b,l,(i,j)}^{(t)}\leq d_{(i,j)}~~~~\forall (i,j)\in D\label{eq:LP-respect-demand}\\
    &{ \lambda_{b,l}^{(t)}}\geq 0~~~~\forall b\in [M],~l\in \mathcal{L}_b,~t\in [n_{b,l}]\label{eq:LP-lambda}
\end{align}
\end{subequations}

\noindent {The reason why (\ref{LP-relaxation}) is a valid LP relaxation is as follows. If (\ref{eq:LP-lambda}) is replaced by $\lambda_{b,l}^{(t)} \in \{0,1\}$, the formulation becomes equivalent to (\ref{eq:ILP-original}). This is because the integrality constraint along with (\ref{eq:LP-one-line-each-bus}) ensure that each bus is assigned exactly one line with one feasible passenger allocation. Note that (\ref{eq:LP-one-line-each-bus}) is set as an equality due to the inclusion of a dummy line in each bus's line set. The resource limits are represented by (\ref{eq:LP-budgets}), while (\ref{eq:LP-respect-demand}) ensures that OD demands are not overserved. Finally, capacity constraints are implicitly respected through the feasible demand allocations \(\theta^{(t)}_{b,l}\). Thus, allowing \(\lambda_{b,l}^{(t)}\) to take continuous values extends the feasible region, thereby forming a valid LP relaxation of (\ref{eq:ILP-original}).}

Note that (\ref{LP-relaxation}) has an exponential number of variables but a polynomial number of constraints. In the next lemma, we will show that it can be solved in polynomial time.

\begin{lemma}\label{lem:LP-relaxation-solvable}
    The LP relaxation \textup{(\ref{LP-relaxation})} can be solved exactly in polynomial time by the ellipsoid method.
\end{lemma}

\begin{proof}[Proof Sketch]
    We apply the ellipsoid method to the dual of \textup{(\ref{LP-relaxation})} together with a separation oracle that, for any dual variable vector $w$, solves for each pair $(b,l)$ the subproblem
    \[
      \max_{\theta_{b,l} \in \mathcal{P}(b,l)} 
      \sum_{(i,j)\in D^{b,l}} \bigl(v_{b,l,(i,j)} - w_{(i,j)}\bigr)\,\theta_{b,l,(i,j)}.
    \]
    As shown in Appendix~\ref{app-sec:lem:LP-solvable}, this is a polynomial-size LP whose constraint matrix is totally unimodular. Hence we obtain an exact polynomial-time separation oracle. Using the ellipsoid method, the primal LP relaxation~\eqref{LP-relaxation} can be solved exactly in time polynomial in the input size (see \cite[Section~4.3]{williamson2011design}).
\end{proof}

\begin{assumption}\label{assumption:hat-n}
Let $\left\{\widehat{\lambda}_{b,l}^{(t)} : b \in [M], l \in \mathcal{L}_b, t \in [n_{b,l}]\right\}$ be an optimal solution to the LP relaxation~\eqref{LP-relaxation} obtained using the ellipsoid method and the separation oracle described above. For each pair $(b, l)$, let $\widehat{n}_{b,l}$ be the number of points $\theta^{(t)}_{b,l}$ such that $\widehat{\lambda}_{b,l}^{(t)} > 0$; without loss of generality we assume that such points are indexed first, that is, $\widehat{\lambda}_{b,l}^{(t)} > 0$ iff $t \in [\widehat{n}_{b,l}]$. By Lemma~\ref{lem:LP-relaxation-solvable}, the indices $[\widehat{n}_{b,l}]$ have polynomial size. This notation will be adopted throughout the paper.
\end{assumption}

\subsection{A Simple Randomized Rounding Algorithm}

Given the optimal solution to (\ref{LP-relaxation}), we propose a randomized rounding algorithm to solve (\ref{eq:ILP-original}). The algorithm first samples a line-configuration pair for each bus according to the LP solution probabilities (Step~\ref{simple-alg-step:sample}), then resolves conflicts where independent sampling over-serves an OD pair by greedily assigning demand in decreasing order of reward value (Step~\ref{simple-alg-step:assign}).

\begin{namedalg}{NC}\label{alg:simple-alg}

~

\noindent {\bf Input:}
        LP relaxation (\ref{LP-relaxation}) and an optimal LP solution $\{\widehat{\lambda}_{b,l}^{(t)}\}$ as in Assumption~\ref{assumption:hat-n}. Let \( \widehat{n}_{b,l} \) be defined as in Assumption \ref{assumption:hat-n}.

\noindent {\bf Output: }A (randomized) solution to (\ref{eq:ILP-original}).

\noindent {\bf Steps: }
\begin{enumerate}
    \item(Sampling) For each $b\in [M]$, a tuple $(l,t)$ (line-configuration pair) is randomly selected (sampled) based on the selection probabilities given by $\widehat{\lambda}_{b,l}^{(t)}$. Denote the selected tuple for each $b\in [M]$ as $(l_b,t_b)$.\label{simple-alg-step:sample}
    \item(Conflict Resolution) Given the sampling realization, for each OD pair $(i,j)\in D$, assume that $v_{1,l_1,(i,j)}\geq v_{2,l_2,(i,j)}\geq \ldots\geq v_{M,l_M,(i,j)}$. Proceed by sequentially assigning $\theta_{b,l_b,(i,j)}^{(t_b)}$ to $b$ operating on line $l_b$, from $b = 1$ to $b = M$. If the total assigned demand for the OD pair $(i,j)$ exceeds $d_{(i,j)}$, discard the excess part that was allocated last, ensuring that the total assigned demand does not exceed $d_{(i,j)}$. \label{simple-alg-step:assign}
\end{enumerate} 

\end{namedalg}



 We will prove the following theorem by demonstrating the superiority of Algorithm \ref{alg:simple-alg} over auxiliary Algorithm \ref{alg:benchmark-nc} that will be described later. 

\begin{theorem}\label{thm:main}
   The demand assignments produced by Algorithm \ref{alg:simple-alg} are integral. Let $\Gamma$ be the optimal objective value of LP relaxation (\ref{LP-relaxation}). The expected reward obtained by Algorithm \ref{alg:simple-alg} is at least  
$$\left(1-\frac{1}{e}\right)\cdot \sum_{b\in [M]}\sum_{l\in \mathcal{L}_b}\sum_{t \in [n_{b,l}]}\sum_{(i,j)\in D^{b,l}}v_{b,l,(i,j)}{\theta}_{b,l,(i,j)}^{(t)}\widehat{\lambda}_{b,l}^{(t)}~=~\left(1-\frac{1}{e}\right)\Gamma.$$
Moreover, the approximation factor $\bigl(1-\tfrac{1}{e}\bigr)$ is tight for LPRC instances without line costs.
\end{theorem}

\begin{proof}
    The approximation ratio is obtained by combining Theorem \ref{thm:benchmark-nc} and Theorem \ref{thm:nc-vs-benchmark} below. The tightness result follows from Theorem~\ref{thm:hardness} in Appendix~\ref{app-sec:thm:hardness}.
\end{proof}

Next we define the auxiliary Algorithm \ref{alg:benchmark-nc}. It is important to note that Algorithm \ref{alg:benchmark-nc} is non-polynomial because it iterates through passengers individually, whereas the standard Algorithm \ref{alg:simple-alg} operates on aggregate passenger counts. However, the former is not intended for implementation, but rather serves as a theoretical construct to prove the performance bound in Theorem \ref{thm:main}.

At the beginning of Algorithm \ref{alg:benchmark-nc}, buses are ranked for each OD pair $(i,j)$ based on their pre-sampling expected rewards $R_{b,(i,j)}$. The algorithm then independently samples a line-configuration pair for each bus using the LP solution probabilities (Step~\ref{bench-step:sampling}), and resolves conflicts by having each bus offer seats to a random subset of passengers of size $\theta_{b,l_b,(i,j)}^{(t_b)}$, with each passenger accepting the highest-ranked bus among those that offered service (Step~\ref{bench-step:conflict}).

\begin{namedalg}{Benchmark NC}\label{alg:benchmark-nc}

~

\noindent \textbf{Input:} An optimal LP solution $\left\{\lambda_{b,\ell}^{(t)}\right\}_{b\in[M],\,\ell,t}$ to \eqref{LP-relaxation}.

\noindent \textbf{Output:} A randomized assignment feasible for \eqref{eq:ILP-original}.

\noindent \textbf{Pre‑computation.} For every OD pair $(i,j)$ and bus $b$ define
\[
  \Lambda_{b,(i,j)}:=\sum_{\ell,t}\lambda_{b,\ell}^{(t)}\,\theta_{b,\ell,(i,j)}^{(t)},
  \qquad
  R_{b,(i,j)}:=\begin{cases}
     \dfrac{\sum_{\ell,t}\lambda_{b,\ell}^{(t)}\,\theta_{b,\ell,(i,j)}^{(t)}\,v_{b,\ell,(i,j)}}{\Lambda_{b,(i,j)}} & \Lambda_{b,(i,j)}>0,\\[4pt]
     0 & \Lambda_{b,(i,j)}=0.
  \end{cases}
\]
Order the buses for each $(i,j)$ in non‑increasing $R_{b,(i,j)}$.

\smallskip

\noindent \textbf{Steps:}
\begin{enumerate}
  \item(Sampling) For each $b\in[M]$, sample a tuple $(\ell_b,t_b)$ independently according to $\lambda_{b,\ell}^{(t)}$.\label{bench-step:sampling}
  \item(Conflict Resolution) For every OD pair $(i,j)$ process its $d_{(i,j)}$ passengers sequentially:\label{bench-step:conflict}
  \begin{enumerate}
    \item For each bus \(b\), it selects uniformly at random exactly \(\theta_{b,\ell_b,(i,j)}^{(t_b)}\) passengers from the \(d_{(i,j)}\) travellers of OD pair \((i,j)\); for each passenger \(p\) of OD pair \((i,j)\), set 
\[
  E_{b,p} = 
  \begin{cases}
    1, & \text{if \(p\) is among the selected set},\\
    0, & \text{otherwise}.
  \end{cases}
\]
    \item Let $B_p:=\{b:E_{b,p}=1\}$, the set of buses that offered service to passenger \(p\) of OD pair \((i,j)\). If $B_p=\varnothing$, leave $p$ unserved.
    Otherwise assign $p$ to the bus in $B_p$ with the highest $R_{b,(i,j)}$. 
  \end{enumerate}
\end{enumerate}

\end{namedalg}

 The next Lemma derives the expected reward generated from each passenger.

\begin{lemma}\label{lem:single-passenger-reward}
    For any passenger $p$ belonging to OD pair $(i,j)$,
\[
  \mathbb{E}[Z_p]
  \;=\;\sum_{b=1}^{M} R_{b,(i,j)}\;x_{b,(i,j)}\prod_{a<b}\bigl(1-x_{a,(i,j)}\bigr),
\]
where $Z_p$ is the reward obtained for~$p$ under Benchmark NC and
$x_{b,(i,j)}:= \Lambda_{b,(i,j)}/d_{(i,j)}\in[0,1]$.
\end{lemma}

\begin{proof}

For any bus \( b \) and OD pair \( (i,j) \), the tuple \( (\ell_b, t_b) \) is sampled with probability \( \lambda_{b,\ell}^{(t)} \). Conditional on this selection, the bus serves exactly \( \theta_{b,\ell_b,(i,j)}^{(t_b)} \) passengers for OD pair \( (i,j) \), chosen uniformly at random from the \( d_{(i,j)} \) total passengers of that pair. Therefore, the conditional probability that a fixed passenger \( p \) from OD pair \( (i,j) \) is selected by bus \( b \) is given by
\[
\binom{d_{(i,j)} - 1}{\theta_{b,\ell,(i,j)}^{(t)} - 1}
\bigg/
\binom{d_{(i,j)}}{\theta_{b,\ell,(i,j)}^{(t)}}
= \frac{\theta_{b,\ell,(i,j)}^{(t)}}{d_{(i,j)}}.
\]
Hence, the probability that bus \( b \) selects passenger \( p \) is
\begin{align}\label{eq:Ebp}
    \Pr(E_{b,p} = 1)
    = \sum_{\ell,t} \lambda_{b,\ell}^{(t)} \cdot \frac{\theta_{b,\ell,(i,j)}^{(t)}}{d_{(i,j)}}
    = \frac{\Lambda_{b,(i,j)}}{d_{(i,j)}}
    = x_{b,(i,j)}.
\end{align}

\noindent Since the buses sample their line–configuration pairs independently, the selection events \( \{E_{b,p}\}_{b=1}^M \) are mutually independent.

With buses ordered so that
$R_{1,(i,j)}\ge R_{2,(i,j)}\ge\cdots\ge R_{M,(i,j)}$,
define
\(
  C_{b,p}:=E_{b,p}\prod_{a<b}(1-E_{a,p}),
\)
i.e.\ $C_{b,p}=1$ iff bus $b$ selects $p$ at Step 2(a) and no higher-ranked bus does.

Let $Z_p$ be the reward earned from passenger~$p$, i.e.
\(
  Z_p=\sum_{b=1}^M C_{b,p}\,v_{b,\ell_b,(i,j)}.
\)
By linearity of expectation,
\begin{align}\label{eq:EZp}
      \mathbb{E}[Z_p]=\sum_{b=1}^M
     \mathbb{E}\bigl[C_{b,p}\,v_{b,\ell_b,(i,j)}\bigr]=\sum_{b=1}^M
      \Pr(C_{b,p}=1)\;
      \mathbb{E}\bigl[v_{b,\ell_b,(i,j)}\mid C_{b,p}=1\bigr].
\end{align}

\noindent Using independence of the $E_{b,p}$’s,
\begin{align}\label{eq:Cbp}
      \Pr(C_{b,p}=1)
  =\Pr(E_{b,p}=1)\prod_{a<b}\Pr(E_{a,p}=0)
  =x_{b,(i,j)}\prod_{a<b}\bigl(1-x_{a,(i,j)}\bigr).
\end{align}

\noindent Moreover, the event $C_{b,p}=1$ implies $E_{b,p}=1$, and  
the choice $(\ell_b,t_b)$ is independent of
other buses’ selections; therefore
\(
  \mathbb{E}\bigl[v_{b,\ell_b,(i,j)}\mid C_{b,p}=1\bigr]
  =\mathbb{E}\bigl[v_{b,\ell_b,(i,j)}\mid E_{b,p}=1\bigr].
\)

\noindent By \eqref{eq:Ebp}, we have that
\[
  \Pr\bigl((\ell_b,t_b)=(\ell,t)\mid E_{b,p}=1\bigr)
  \;=\;
  \frac{\lambda_{b,\ell}^{(t)}\,\theta_{b,\ell,(i,j)}^{(t)}/d_{(i,j)}}
       {\Lambda_{b,(i,j)}/d_{(i,j)}}
  \;=\;
  \frac{\lambda_{b,\ell}^{(t)}\,\theta_{b,\ell,(i,j)}^{(t)}}
       {\Lambda_{b,(i,j)}}.
\]
Therefore
\begin{align}\label{eq:cond-v-Ebp}
      \mathbb{E}[\,v_{b,\ell_b,(i,j)}\mid E_{b,p}=1]
  \;=\;
  \sum_{\ell,t}
     v_{b,\ell,(i,j)}
     \frac{\lambda_{b,\ell}^{(t)}\,\theta_{b,\ell,(i,j)}^{(t)}}
          {\Lambda_{b,(i,j)}}
  \;=\;R_{b,(i,j)}.
\end{align}

Substituting \eqref{eq:Cbp} and \eqref{eq:cond-v-Ebp} into \eqref{eq:EZp}, we finish the proof.
\end{proof}

We are now in a position to establish the performance bound of Algorithm~\ref{alg:benchmark-nc}.

\begin{theorem}\label{thm:benchmark-nc}
Let\/ $\Gamma$ denote the optimal objective value of the LP relaxation
\eqref{LP-relaxation}.  
The expected total reward produced by
Algorithm~\ref{alg:benchmark-nc} satisfies
\[
     \mathbb{E}\bigl[\textup{reward of Algorithm~\ref{alg:benchmark-nc}}\bigr]
     \;\;\ge\;\;
     \Bigl(1-\frac1e\Bigr)\,\Gamma .
\]
\end{theorem}

\begin{proof}

Fix an OD pair $(i,j)\in D$ and a passenger
$p$ of $(i,j)$.  
With the notation of Algorithm \ref{alg:benchmark-nc}, 
 Lemma \ref{lem:single-passenger-reward} proves that
\(
   \E[Z_p]\;=\;
   \sum_{b=1}^{M}
     R_{b,(i,j)}\,x_{b,(i,j)}
     \prod_{a<b}\bigl(1-x_{a,(i,j)}\bigr),
\)
where the buses are ordered so that
$R_{1,(i,j)}\ge R_{2,(i,j)}\ge\cdots\ge R_{M,(i,j)}$ and
$x_{b,(i,j)}\in[0,1]$ with $\sum_{b}x_{b,(i,j)}\le1$. Then due to
\cite[Lemma 2.1 and Proof of Lemma 2.2]{fleischer2011tight}, we have that
\begin{equation}\label{eq:passenger}
   \E[Z_p]\;=\;
   \sum_{b=1}^{M}
      R_{b,(i,j)}\,x_{b,(i,j)}\prod_{a<b}(1-x_{a,(i,j)})
   \;\;\ge\;\;
   \Bigl(1-\frac1e + \frac{1}{32M^2}\Bigr)\sum_{b=1}^{M}R_{b,(i,j)}\,x_{b,(i,j)}.
\end{equation}

The total reward is
$\sum_{(i,j)\in D}\sum_{p \text{ from } (i,j)} Z_p$; taking expectations and
applying~\eqref{eq:passenger} yields
\[
\begin{aligned}
  \E[\text{total reward}]
  &= \sum_{(i,j)\in D}\sum_{p \text{ from } (i,j)} \E[Z_p]
  \ge \Bigl(1-\tfrac1e+\frac{1}{32M^2}\Bigr)
       \sum_{(i,j)\in D}\sum_{p \text{ from } (i,j)}\sum_{b=1}^{M}
         R_{b,(i,j)}\,x_{b,(i,j)} \\[-2pt]
  &= \Bigl(1-\tfrac1e+\frac{1}{32M^2}\Bigr)
       \sum_{(i,j)\in D}\sum_{b=1}^{M}
         R_{b,(i,j)}\,x_{b,(i,j)}\,d_{(i,j)}
    = \Bigl(1-\tfrac1e+\frac{1}{32M^2}\Bigr)
       \sum_{(i,j)\in D}\sum_{b=1}^{M}
       R_{b,(i,j)}
       \Lambda_{b,(i,j)}\\[2pt]
  &= \Bigl(1-\tfrac1e+\frac{1}{32M^2}\Bigr)
       \sum_{b=1}^{M}\sum_{\ell,t}\sum_{(i,j)\in D}
       \lambda_{b,\ell}^{(t)}\,
       \theta_{b,\ell,(i,j)}^{(t)}\,
       v_{b,\ell,(i,j)},
\end{aligned}
\]
The term in the last equality equals $\bigl(1-\frac{1}{e}+\frac{1}{32M^2}\bigr)$ times the objective value of the given LP solution. Then, by direct substitution and calculation, we complete the proof.
\end{proof}

 Algorithm \ref{alg:simple-alg} assigns passengers to buses based on the post-sampling (realized) reward values (and therefore the allocation is the best possible under the realization), whereas Algorithm \ref{alg:benchmark-nc} assigns passengers to buses based on pre-sampling (expected) reward values. Hence, for any sampling realization, the allocation in Algorithm NC cannot be worse than the allocation in Algorithm \ref{alg:benchmark-nc}. Consequently, Algorithm \ref{alg:simple-alg} inherits the performance guarantee established for Algorithm \ref{alg:benchmark-nc}.

\begin{theorem}
\label{thm:nc-vs-benchmark}
Fix any realization of the sampling step  
$\bigl\{(l_b,t_b)\bigr\}_{b\in[M]}$  
that is common to Algorithm~\ref{alg:simple-alg} and
Algorithm~\ref{alg:benchmark-nc}.
Conditional on this sampling outcome,
the total reward produced by Algorithm~\ref{alg:simple-alg} is \emph{at least} the total reward
produced by Algorithm~\ref{alg:benchmark-nc} regardless of its additional internal randomness.  
Consequently,
\(
   \E \bigl[\text{reward of NC}\bigr]
   \;\;\ge\;\;
   \E
   \bigl[\text{reward of Algorithm~\ref{alg:benchmark-nc}}\bigr].
\)
\end{theorem}

\begin{proof}
Assume the two algorithms share the same sampled tuples
$(l_b,t_b)$, then
\(
   \theta_{b}:=\theta_{b,l_b}^{(t_b)}
\)
are identical for both algorithms.
We compare the two assignments OD pair by OD pair.

Fix \((i,j)\in D\) and abbreviate 
\(
   v_{b}:=v_{b,l_b,(i,j)}, 
   \text{ and }
   \theta_{b}:=\theta_{b,l_b,(i,j)}^{(t_b)}.
\)
Let \(N^\star_{b}\) be the number of passengers assigned to bus \(b\) by Algorithm \ref{alg:simple-alg}.  The algorithm orders the buses so that
\(
   v_{1}\ge v_{2}\ge\cdots\ge v_{M},
\)
and then greedily sets
\(
   N^\star_{b} \;:=\;\min\Bigl\{\theta_{b},\;d_{(i,j)}-\sum_{a<b}N_{a}\Bigr\},
   \text{ for each } b=1,\dots,M.
\)
This exactly solves the linear knapsack problem
\[
   \max\;\sum_{b=1}^{M} v_{b}\,N_{b}
   \quad\text{s.t.}\quad
   0\le N_{b}\le\theta_{b}\quad\forall b,\quad
   \sum_{b=1}^{M}N_{b}\le d_{(i,j)},
\]
and thus constitutes an \emph{optimal} solution to the knapsack whose objective matches the total reward for that OD pair under the sampled lines.

By contrast, Algorithm~\ref{alg:benchmark-nc} first selects \(\theta_{b}\) passengers for each bus \(b\), then assigns each chosen passenger to one bus.  Let \(\widetilde N_{b}\) be the number of passengers ultimately boarded on \(b\) by Algorithms~\ref{alg:benchmark-nc}.  Clearly
\(
   0\;\le\;\widetilde N_{b}\;\le\;\theta_{b}\),
   and
   \(\sum_{b=1}^{M}\widetilde N_{b}\;\le\;d_{(i,j)},
\)
but since Algorithm \ref{alg:simple-alg} attains the knapsack optimum,
\(
   \sum_{b=1}^{M}v_{b}\,\widetilde N_{b}
   \;\le\;
   \sum_{b=1}^{M}v_{b}\,N^\star_{b}.
\)

Since this holds for every sampling realization and each OD pair, summing over all \((i,j)\) and then taking expectations over both the common sampling step and Algorithm \ref{alg:benchmark-nc}'s internal randomness shows that Algorithm \ref{alg:simple-alg} achieves at least as high an expected reward as Algorithm \ref{alg:benchmark-nc}.
\end{proof}

\section{Extensions to Resource-Constrained Settings}\label{sec:extensions}

While Algorithm NC provides optimal approximation guarantees for the resource-unconstrained case, the general LPRC problem includes multiple resource constraints. We show how Algorithm NC extends to handle resource constraints using decomposition techniques in \cite{lee2009non}. To maintain the flow of the main text, we present a high-level overview of the results here. The detailed algorithms and their corresponding proofs are provided in Appendices \ref{sec:approx-LPRC} and \ref{sec:tolerance}.

\subsection{General-Cost Extension}

When line costs can vary significantly across different resource types, the challenge is to handle this heterogeneity while maintaining approximation guarantees. Some lines may have small costs across all resources while others may be expensive in at least one resource type.
We develop a polynomial-time approximation algorithm that addresses this challenge through a heavy/light decomposition approach.

\begin{corollary}[General Case]\label{cor:general}
For arbitrary costs, Algorithm \ref{alg:approx-all-costs} (see Appendix \ref{sec:approx-LPRC}) achieves the expected approximation ratio $(\frac{1}{2}-\frac{1}{2e}-\eta)$ for any given constant $\eta\in (0,1/4)$.
\end{corollary}

\textbf{Approach:} For a specially chosen constant $\delta$, the algorithm partitions the set of lines into low-cost (all costs $\le \delta$) and high-cost (some costs $>\delta$) subsets, then enumerates all feasible assignments for high-cost lines (polynomially many due to resource constraints). It compares the LP-relaxation value for the problem limited to only the low-cost lines against the best LP value among enumerated high-cost assignments, applying the appropriate rounding to the better case. The key insight is that any feasible solution can be decomposed into low-cost and high-cost parts, so the maximum of our two approaches achieves at least half the optimal value. For the low-cost case, we apply Algorithm LC (detailed in Appendix \ref{sec:low-cost}), which uses Algorithm NC with modified sampling probabilities. Small costs ensure this maintains near-optimal performance through concentration inequalities. In the next section we show that allowing a small resource augmentation eliminates the $\frac{1}{2}$ factor and yields a near-optimal guarantee.

\subsection{Resource-Tolerant Extension}

Many operational settings can tolerate a small degree of excess resource consumption, such as a minor increase in spending beyond the allocated budget. This flexibility can significantly improve approximation ratios. 
We develop a polynomial-time algorithm that achieves near-optimal approximation ratios under this augmented resource model.

\begin{corollary}[Resource-Tolerant Case]\label{cor:tolerant}
Algorithm C-Tol (see Appendix \ref{sec:tolerance}) achieves expected approximation ratio $(1-\frac{1}{e}-\eta)$ while using at most $(1+\tau)$ units of each resource, for any given constants \(\eta, \tau \in \left(0, \frac{1}{2}\right)\).
\end{corollary}

\textbf{Approach:} The algorithm enumerates high-cost assignments as in the general case, but for each enumerated assignment, it solves a modified LP where high-cost assignments are fixed, remaining resources are augmented by $\tau$, and low-cost line costs are scaled appropriately. It then applies Algorithm LC (Appendix \ref{sec:low-cost}) to the best modified LP solution. The key insight is that resource augmentation ensures the modified LP upper bounds the original optimal value, while careful scaling maintains the small-cost conditions needed for Algorithm LC. This enables near-optimal rounding, with the $(1+\tau)$ resource usage following directly from the augmentation structure.

\begin{remark}
   The approximation ratio in Corollary \ref{cor:tolerant} is nearly tight by Theorem \ref{thm:hardness}. 
\end{remark}

\section{A Practical Algorithm}\label{sec:practical}

To illustrate that this framework can combine solid theoretical guarantees with good practical performance, this section develops an algorithm tailored for real-world applications. Building on the core randomized rounding strategy of Algorithm \ref{alg:simple-alg}, we introduce several crucial modifications to improve empirical performance. These enhancements are validated through numerical simulations in Section \ref{sec:computation-advance} and Appendix \ref{sec:app-synthetic}, where we demonstrate that the proposed algorithm is computationally effective, particularly on large-scale instances where contemporary solvers like Gurobi fail due to high memory and computational requirements.

First, to reduce the problem size for large fleets, the algorithm leverages the observation that buses can often be classified into homogeneous groups based on shared characteristics (e.g., capacity and fuel type). This allows for an aggregation of variables in the LP relaxation, significantly reducing its size without loss of optimality.
Second, to manage resource constraints (which are absent in the setting for Algorithm \ref{alg:simple-alg}), we introduce a scaling parameter~$\epsilon \in (0,1)$ that reduces the selection probabilities during randomization. Rather than using a fixed value, our approach includes a novel method for tuning~$\epsilon$ based on instance-specific characteristics to maximize the theoretical performance guarantee.
Third, if a randomized solution violates resource constraints, a repair heuristic is applied to restore feasibility. This procedure intelligently removes bus assignments to preserve the valuable parts of a near-feasible solution rather than dismissing it entirely.
Finally, after a feasible solution is found, a post-processing phase is introduced to utilize any remaining bus capacity and serve unassigned passengers.

The first modification reduces the problem size by aggregating buses with identical characteristics. Proposition \ref{prop:aggregate} establishes the validity of this approach.


\begin{prop}\label{prop:aggregate}
Let the set of all buses $[M]$  be partitioned into $W$ disjoint subsets $\mathcal{M}_w \subseteq [M], w \in [W]$. All the buses in the same subset are homogeneous. In particular, if buses $b$ and $b'$ belong to the same subset $\mathcal{M}_w$, then they share the same candidate line sets $(\mathcal{L}_w:=\mathcal{L}_b =\mathcal{L}_{b'})$ and have identical parameters. Hence, for buses from the same $\mathcal{M}_w$ we use index $w$ instead of $b$:  $D^{w,l} = D^{b,l}$, $v_{w,l,(i,j)} = v_{b,l,(i,j)}$, $c_{w,l}^{(k)}=c_{b,l}^{(k)}$, $n_{w,l} = n_{b,l}$ and $\theta_{w,l,(i,j)}^{(t)}=\theta_{b,l,(i,j)}^{(t)}$ for all $b\in \mathcal{M}_w$. Then, consider the following LP problem:
\begin{subequations}\label{LP-relaxation-aggregated}
\begin{align}
\max_{\lambda\geq 0}~~~~~~~~~~&\sum_{w\in [W]}\sum_{l\in \mathcal{L}_w}\sum_{t \in [n_{w,l}]}\sum_{(i,j)\in D^{w,l}}|\mathcal{M}_w|\cdot v_{w,l,(i,j)}{\theta}_{w,l,(i,j)}^{(t)}\lambda_{w,l}^{(t)}\\
s.t.~~~~~~~~~~~~& \sum_{l\in\mathcal{L}_w}\sum_{t \in [n_{w,l}]}\lambda_{w,l}^{(t)} = 1~~~~\forall w\in [W]\\
&\sum_{w\in [W]}\sum_{l\in \mathcal{L}_w}\sum_{t\in [n_{w,l}]} |\mathcal{M}_w|\cdot c_{w,l}^{(k)}\cdot \lambda_{w,l}^{(t)}\leq 1~~~~\forall k\in [K]\\
&\sum_{w\in [W]}\sum_{l\in \mathcal{L}_w}\sum_{t \in [n_{w,l}]}|\mathcal{M}_w|\cdot \lambda_{w,l}^{(t)}{\theta}_{w,l,(i,j)}^{(t)}\leq d_{(i,j)}~~~~\forall (i,j)\in D
\end{align}
\end{subequations}  
\noindent Let $\left\{\widehat{\lambda}_{w,l}^{(t)}: w\in [W], l\in \mathcal{L}_w, t\in [n_{w,l}]\right\}$ be an optimal solution to (\ref{LP-relaxation-aggregated}). Then we can derive an optimal solution to (\ref{LP-relaxation}) by letting $\lambda_{b,l}^{(t)} :=\widehat{\lambda}_{w,l}^{(t)}$ for $b\in \mathcal{M}_w, w\in [W], l\in \mathcal{L}_w, t\in [n_{w,l}]$. 

\end{prop}

\begin{proof}
It is straightforward to show through direct calculations that the derived solution to (\ref{LP-relaxation}) is feasible and has the same objective value as an optimal solution to (\ref{LP-relaxation-aggregated}). Regarding optimality, let  $\left\{\widetilde{\lambda}_{b,l}^{(t)}: b\in [M], l\in \mathcal{L}_b, t\in [n_{b,l}]\right\}$ be an optimal solution to (\ref{LP-relaxation}).  Construct a new solution  $\left\{\widehat{\lambda}_{w,l}^{(t)}: w\in [W], l\in \mathcal{L}_w, t\in [n_{w,l}]\right\}$  to (\ref{LP-relaxation-aggregated}) as 
$\widehat{\lambda}_{w,l}^{(t)} = \frac{1}{|\mathcal{M}_w|} \sum_{b\in \mathcal{M}_w} \widetilde{\lambda}_{b,l}^{(t)}$. 
Due to the homogeneity of buses within groups and the linearity of the objective and constraints, this new solution remains feasible and preserves the objective value. Therefore, the optimal objective values of (\ref{LP-relaxation}) and  (\ref{LP-relaxation-aggregated})  are equal.
\end{proof}

\begin{remark}
In practice $W$ is typically small with respect to $M$ as there are not many different types of buses, hence (\ref{LP-relaxation-aggregated}) is typically much smaller than (\ref{LP-relaxation}). This simplification helps practical application of this approach since it is based on solving LP relaxation (\ref{LP-relaxation}). We note that  the integrality and multi-dimensional capacity constraints prevent a similar simplification for the ILP formulation (\ref{eq:ILP-original}).
\end{remark}
Now we describe the proposed algorithm for LPRC. Algorithm~\ref{alg:low-cost-modified} builds upon Algorithm~\ref{alg:simple-alg} but introduces a scaling parameter $\epsilon \in (0,1)$ to handle resource constraints (Step~\ref{alg-step:LC-M-sample}), resolves demand conflicts as before (Step~\ref{alg-step:LC-M-rounding}), repairs any resource constraint violations via a heuristic (Step~\ref{alg-step:LC-M-remove}), and assigns unserved passengers to residual bus capacity in a post-processing phase (Step~\ref{alg-step:LC-M-residual}). Steps~\ref{alg-step:LC-M-remove} and~\ref{alg-step:LC-M-residual} are intended to enhance practical performance and do not alter the theoretical analysis.

\begin{namedalg}{PR}\label{alg:low-cost-modified}

~

\noindent {\bf Input:}
LP relaxation (\ref{LP-relaxation-aggregated}), an optimal LP solution $\left\{\widehat{\lambda}_{w,l}^{(t)}:w\in [W], l\in \mathcal{L}_w, t\in [\widehat{n}_{w,l}]\right\}$, and $\{\mathcal{M}_w\}_{w\in [W]}$.

\noindent {\bf Output: }A (randomized) solution to (\ref{eq:ILP-original}).

\noindent {\bf Steps: }
\begin{enumerate}
\item
Set $\epsilon \in \left(0,1\right)$. (Strategy for choosing $\epsilon$ will be discussed later.) For each $b\in [M]$, if $b\in \mathcal{M}_w$, sample a tuple $(l,t)$ with the selection probability given by $(1-\epsilon)\widehat{\lambda}_{w,l}^{(t)}$. Denote the sampled tuple for each $b\in [M]$ as $(b,l_b,t_b)$. Note that the total of the selection probabilities for bus $b$ is less than 1. If no tuple is selected, the bus is not used (assigned to a ``dummy" line with zero costs and rewards).\label{alg-step:LC-M-sample}

\item Same as Step \ref{simple-alg-step:assign} in Algorithm \ref{alg:simple-alg}. \label{alg-step:LC-M-rounding}

\item If the obtained solution meets the resource constraints, it becomes the temporary solution. Otherwise, some heuristic is applied to unselect some of the selected buses until all resource constraints are met (we describe Heuristic H1 that we used for this purpose in Appendix \ref{sec:heuristics}), and the obtained solution becomes the temporary solution. \label{alg-step:LC-M-remove}

\item 
Compute the residual OD demands $\tilde d_{(i,j)}$ and the residual arc capacities
$r_{b,l_b}(e)$ for every selected bus--line pair $(b,l_b)$ after
Step~\ref{alg-step:LC-M-remove}.
Assign unserved passengers to these
already-selected buses using Heuristic~H2 defined in Appendix \ref{sec:heuristics}.
This postprocessing never increases any resource usage and preserves feasibility.
\label{alg-step:LC-M-residual}
\end{enumerate}

\end{namedalg}

\noindent Notably, we have the following theoretical guarantee, which holds regardless of the details of the heuristics used in Steps~\ref{alg-step:LC-M-remove} and~\ref{alg-step:LC-M-residual}:

\begin{prop}\label{prop:alg-R-approx}
    Let   $J$ be the largest  integer satisfying: 
$
\max_{b\in [M],l\in \mathcal{L}_b,k\in [K]}c_{b,l}^{(k)} \leq  \frac{1}{J}
$, and $\delta := \frac{\epsilon}{1-\epsilon}$.
Let $\delta_i := \delta +\frac{i-J}{J\cdot (1-\epsilon)}$, and $\widehat{J}$ be the smallest integer greater than or equal to $J$ such that $\delta_{\widehat{J}} \geq 1$. Then Algorithm \ref{alg:low-cost-modified} achieves an expected approximation ratio of at least: 
\begin{align*}
    \alpha:=\left(1-\frac{1}{e}\right)\cdot \left(1-\epsilon\right)\cdot \Bigg(1-K\cdot \left(\frac{e^{\delta}}{(1 + \delta)^{1 + \delta}}\right)^{J \cdot (1 - \epsilon)} 
 - \frac{K}{J}\cdot \Bigg(&\sum_{i= J}^{\widehat{J}-1} \left(\frac{e^{\delta_i}}{(1 + \delta_i)^{1 + \delta_i}}\right)^{J \cdot (1 - \epsilon)}+4e^{-(\epsilon\cdot {J} + \widehat{J} - J) /3}
\Bigg)\Bigg).
\end{align*}
\noindent In particular, let $K = 3$.  Table \ref{table:transposed_epsilon_alpha_J} provides values $\alpha$ for some values of $J$ and $\epsilon$. 
\begin{table}[H]
\centering
\begin{tabular}{|c|c|c|c|c|c|c|c|c|}
\hline
\(J\) & 10 & 15 & 20 & 30 & 40 & 45 & 60 & 80 \\ \hline
\(\epsilon\) & 0.62 & 0.54 & 0.49 & 0.42 & 0.38 & 0.37 & 0.33 & 0.29 \\ \hline
\(\alpha\) & 0.181 & 0.250 & 0.291 & 0.338 & 0.367 & 0.378 & 0.404 & 0.428 \\ \hline
\end{tabular}
\caption{Values \(\alpha\) for several \(J\).}
\label{table:transposed_epsilon_alpha_J}
\end{table}



\end{prop}



\noindent The proof of Proposition \ref{prop:alg-R-approx}  can be found in Appendix \ref{app-sec:prop:alg-R-approx-proof}.   
The analysis in Proposition~\ref{prop:alg-R-approx} and the examples in Table~\ref{table:transposed_epsilon_alpha_J} show that the optimal~$\epsilon$ varies based on the instance parameter~$J$. This motivates our practical approach to sample~$\epsilon$ uniformly from the range $[0.01, 0.6]$ to explore the solution space effectively.

{\section{Numerical Experiments}\label{sec:computational-exp}

In this section, we present results of computational experiments based on the data provided in  \cite{bertsimas2021data}. The data uses the road network of the greater Boston area. Trip requests data are obtained from  \cite{mbta2014}. The total demand consists of $93,826$ passenger requests, with $597$ OD pairs. For more detailed information, please refer to \cite{bertsimas2021data}. We also conduct numerical experiments on synthetic data to further demonstrate the computational efficiency of our algorithms. These results are presented in the Appendix \ref{sec:app-synthetic}.

For each OD pair $(i,j)$, let $D_s$ denote the shortest distance between origin $i$ 
and destination $j$, $D_l$ the distance traveled on line $l$ from boarding to alighting 
stop (optimized over all feasible stop pairs on $l$), and $D_{W1}$, $D_{W2}$ the 
walking distances from the origin to the nearest stop on $l$ and from the nearest stop 
to the destination, respectively. We define the reward as 
\begin{align}\label{eq:reward}
v_{b,l,(i,j)} = \max\!\left\{0,\; 
\frac{2.6\cdot D_s - m_b \cdot D_l - 6\cdot (D_{W1}+D_{W2})}{D_s}\right\},
\end{align}
where the factor $2.6$ excludes routes with detours exceeding typical maximum urban 
circuity indices \cite[Section~4.1.1]{yang2020path}, the factor $6$ reflects that 
walking is roughly six times slower than bus travel, and $m_b \geq 1$ is a 
bus-type-dependent multiplier that captures the extra time cost due to more frequent 
passenger boarding and alighting on larger-capacity buses (see, e.g., 
\cite{vijayakumar1990analysis}). Note that this reward specification can be readily 
adjusted to incorporate other user preferences (e.g., road quality along the route) and 
alternative system configurations (e.g., multi-modal first- and last-mile access).

Unless otherwise stated, all experiments use the following heterogeneous-fleet setting. 
We consider buses with capacities of 30, 40, and 50 passengers, evenly distributed 
among the fleet (e.g., a 60-bus fleet would have 20 buses of each capacity). The 
bus-type multiplier in \eqref{eq:reward} is set to $m_b \in \{1.0, 1.05, 1.1\}$ for 
capacities $30$, $40$, and $50$, respectively. Additionally, we consider two fuel 
types: traditional and electric. Within each capacity, half the buses use traditional 
fuel and half are electric. In this setting, we can partition the candidate bus set into 
$6$ groups, i.e.\ $W = 6$ in \eqref{LP-relaxation-aggregated}. For each of the $6$ 
groups, a different set of $3000$ candidate lines is generated according to the 
methodology from \cite{silman1974planning}. All buses within the same group share a 
common set of $3000$ candidate lines.


For each bus $b$ and each line $l$ in its candidate line set, we define three types of costs. The first cost, directly proportional to the line's length, is denoted by $T_l$ and is called Distance Cost. The Acquisition Cost for the bus is set at $2\cdot \sqrt{C_b}$ for electric buses and $\sqrt{C_b}$ for buses using traditional fuel, reflecting the higher purchase prices associated with electric vehicles (\cite{maloney2019electric}). The Emission Cost, dependent on the bus's capacity and the line's length, is defined as $T_l\cdot \sqrt{C_b}$ for traditional fuel buses, but is reduced to $3\cdot T_l\cdot \sqrt{C_b}/10$ for electric buses to account for their much lower environmental impact. Resource limits are normalized to one, with costs scaled by their corresponding budgets.

For the rest of this section, we present a series of computational experiments to validate our proposed model and approach. First, we demonstrate the superiority of heterogeneous fleets over homogeneous alternatives. Next, we establish the necessity of multi-resource optimization by highlighting the limitations of single-constraint models. Thirdly, we assess the scalability of our approach on instances with different scales to confirm its practical applicability. Finally, we conduct an ex-post analysis of passenger preferences to demonstrate that our assignments align well with individual rider preferences.

\subsection{The Superiority of a Heterogeneous Fleet}\label{sec:heter}

A core feature of our LPRC model is its ability to optimize a \textbf{heterogeneous fleet} by composing a strategic mix of different bus types to maximize the overall system performance. To quantify the benefits of our approach, we compared our heterogeneous-fleet solution against six distinct \textbf{homogeneous fleet} configurations. The solutions for the homogeneous fleets were obtained by solving their respective integer programs to near-optimality (1\% MIP gap) using Gurobi. For each \textbf{heterogeneous case}, in contrast, the solution was found by obtaining the LP relaxation and then performing 3,000 independent runs of Algorithm \ref{alg:low-cost-modified}. The best-performing solution was then selected.

The experimental design is as follows. The heterogeneous fleet consists of 30 buses with at most 5 units per type, where each type has access to its own dedicated set of 3,000 candidate lines. Each homogeneous fleet, by contrast, deploys 30 buses of a single type with access to the combined pool of all 18,000 candidate lines. This design \emph{favors the homogeneous fleet} in two ways. First, it permits up to 30 buses of a single type rather than the cap of 5. Second, it provides each bus access to the full combined line pool rather than a single type-specific subset. The performance advantage of the heterogeneous fleet therefore reflects a genuine benefit of fleet diversification rather than any other structural advantage in the experimental design. The experiments were conducted under three distinct budget scenarios representing small, medium, and large levels of resource availability. Specifically, we scaled the budget parameters to make resource constraints progressively less restrictive. This allowed the heterogeneous fleet to deploy approximately 10, 15, and 20 buses under the small, medium, and large budget scenarios, respectively.

\begin{figure}[htbp]
    \centering    \includegraphics[width=1\linewidth]{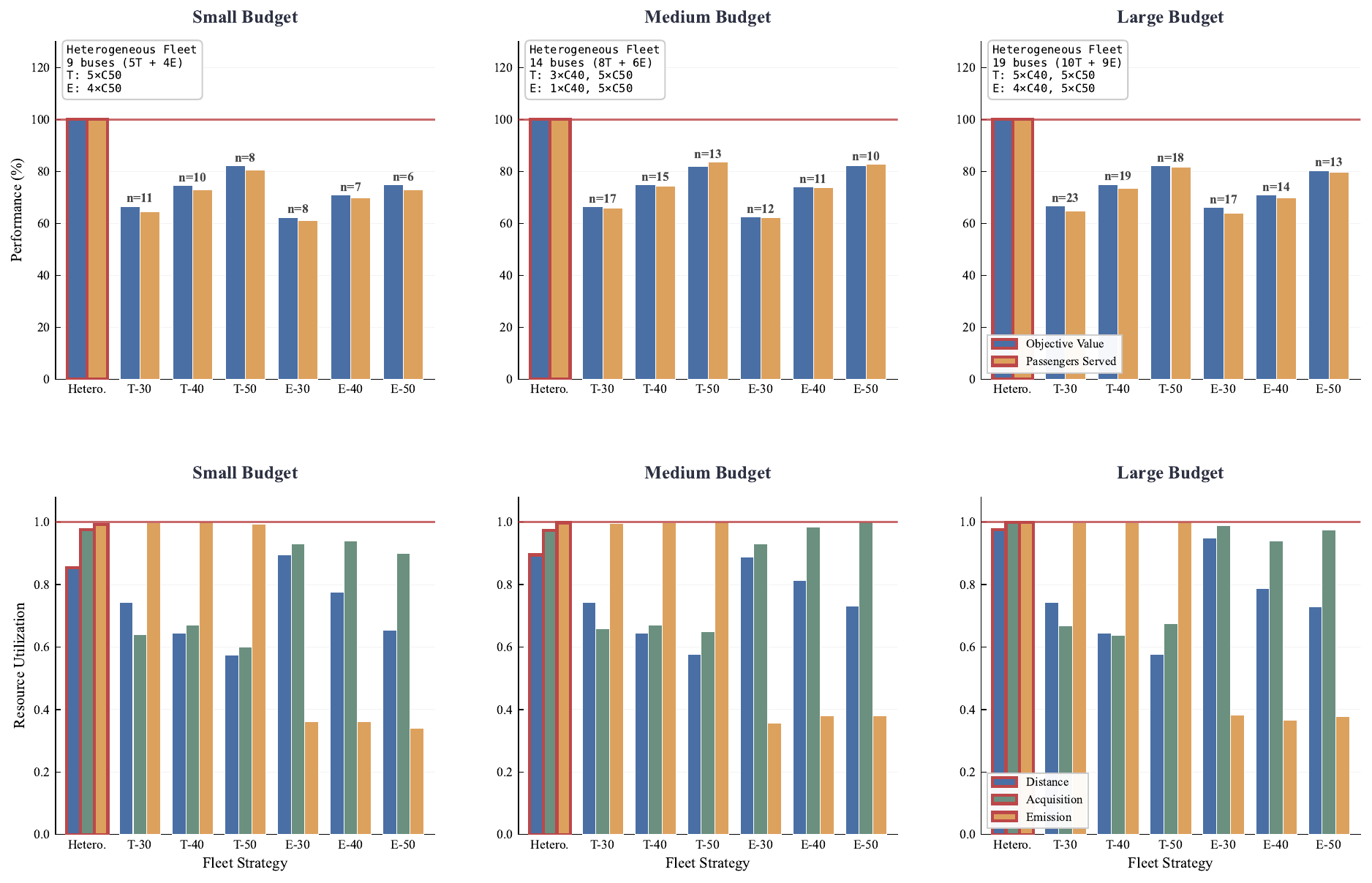}
    \caption{Performance comparison and resource utilization of heterogeneous versus homogeneous fleets across small, medium, and large budget scenarios. }
    \label{fig:het_adv}
\end{figure}

Figure~\ref{fig:het_adv} compares the performance and resource utilization of heterogeneous fleets against homogeneous fleets across three budget scenarios. Each column corresponds to a budget level (small, medium, large), with performance metrics above and resource utilization below. The x-axis labels denote fleet strategies: ``Hetero.'' for heterogeneous fleets, ``T-30/40/50'' for homogeneous traditional bus fleets with capacities of 30, 40, and 50 passengers, and ``E-30/40/50'' for homogeneous electric bus fleets with the same capacities. The annotation $n$ above each bar indicates the number of buses selected in that fleet. The text box in each performance panel details the heterogeneous fleet composition.

In the \textbf{small budget} scenario (first column of Figure~\ref{fig:het_adv}), the advantage of a heterogeneous fleet is particularly evident. It assembles a lean, highly optimized fleet of just 9 buses composed of five traditional and four electric vehicles, all with a capacity of 50 passengers. This composition is achieved by carefully balancing resource consumption, utilizing 85.5\% of the distance budget, 97.5\% of the acquisition budget, and 99.2\% of the emission budget. In stark contrast, the optimal Gurobi solutions for every homogeneous fleet strategy fall significantly short, achieving objective values that are 18\% to 38\% lower. The key reason for their underperformance is inefficient resource management (see Figure \ref{fig:het_adv}, bottom of column 1). Uniform electric fleets are constrained by high acquisition costs, leaving other resources underutilized. Conversely, traditional fleets are bottlenecked by emission limits, which restricts access to high-value routes. This performance gap persists even as resources become more plentiful (second and third columns of Figure~\ref{fig:het_adv}). With medium and large budgets, homogeneous strategies continue to lag by 18\% to 38\%. This consistent performance gap across all budget levels demonstrates the benefits of fleet heterogeneity.

We additionally evaluate a 60-bus fleet under proportionally equivalent budget scenarios. The results are reported in Figure~\ref{fig:het_adv_60bus} in Appendix~\ref{sec:app-het}, which show the same pattern as in the 30-bus case, with heterogeneous fleets delivering a significant advantage.

\subsection{The Importance of Multi-Resource Optimization}\label{sec:multi-resource-advantage}

Having established the advantages of fleet heterogeneity, we now examine whether multi-resource optimization is really necessary, or could be replaced with single-resource optimization. Using the same heterogeneous fleet throughout, we compared our \textbf{full multi-resource-constraint model} against several \textbf{single-resource-constraint models}. The multi-constraint solution was generated using Algorithm \ref{alg:low-cost-modified}, while the single-constraint solutions were found using Gurobi (1\% MIP gap).  We evaluate two simplified single-resource-constraint models: (i) one that enforces exactly one resource as a hard constraint and (ii) one that averages the normalized costs across all three resources with equal weights into a single constraint. 

In single-resource models, we penalize resource costs in the objective using coefficients small enough to discourage waste without altering the optimal solution.

\begin{figure}[h]
    \centering
    \includegraphics[width=\linewidth]{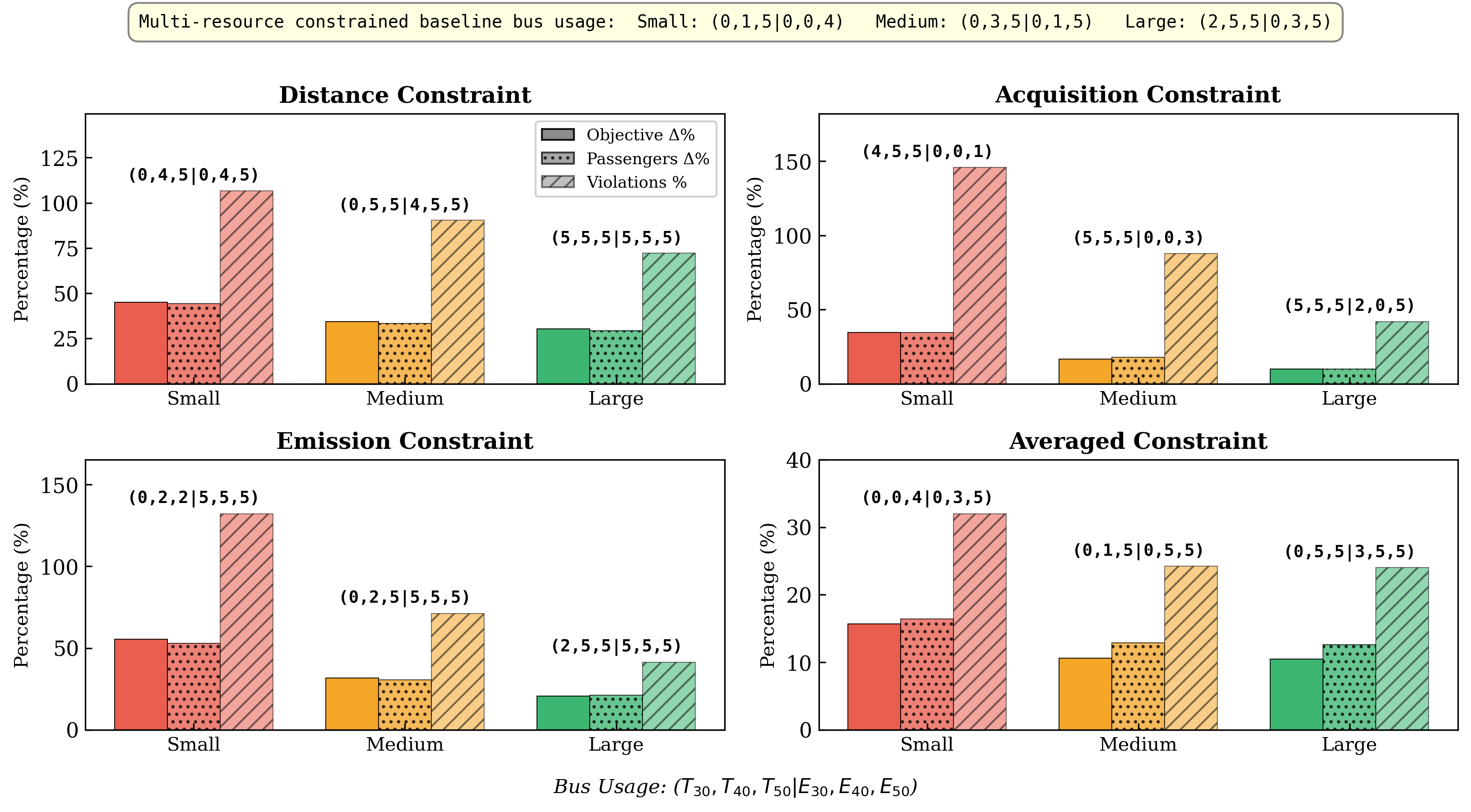}
    \caption{Performance comparison of single-resource optimization strategies across budget levels (30-bus).}
    \label{fig:strategy_comparison}
\end{figure}

We first evaluate a 30-bus heterogeneous fleet under the same setup of Section~\ref{sec:heter}, but with adjusted budgets across all three scenarios to avoid trivial solutions from single-resource benchmarks. Each panel in Figure~\ref{fig:strategy_comparison} represents a different constraint strategy, with bars showing the percentage increase in objective value (solid) and passenger coverage (dotted) relative to the multi-resource baseline, along with total budget violations (diagonal lines). Bus usage is annotated above each group as $(T_{30},T_{40},T_{50}|E_{30},E_{40},E_{50})$. The baseline box displays the feasible fleet configuration for each budget level derived from Algorithm~\ref{alg:low-cost-modified}.

The results reveal a consistent trade-off between objective improvement and constraint violation. Every single-resource strategy achieves higher objective values but at the cost of substantial constraint violations. At the Small budget level, the Emission, Distance, and Acquisition constraints yield objective gains of 55.3\%, 45.1\%, and 34.4\%, but incur total violations of 132.2\%, 106.7\%, and 145.7\%, respectively. These magnitudes decrease with larger budgets but the pattern persists. In particular, at the Large level, gains ranging from 9.9\% to 30.2\% still produce violations of 41.4\% to 72.2\%. This occurs simply because optimizing under a single resource exhausts the unconstrained ones. 

The Averaged approach, which aggregates all resources into a single weighted constraint, reduces total violations to between 24.0\% and 32.0\% but yields only 10.5\% to 15.7\% objective improvement.

We additionally evaluate a 60-bus fleet under the similar setup as the 30-bus experiment. The results are reported in Figure~\ref{fig:strategy_comparison_60} in Appendix \ref{sec:app-single-resource}, which show the same pattern as in the 30-bus case.

\subsection{Computational Advancements}\label{sec:computation-advance}

The previous experiments established the advantages of heterogeneous fleets and the necessity of multi-resource optimization. We now examine the scalability of Algorithm~\ref{alg:low-cost-modified}, which produced the benchmark solutions in those experiments. Specifically, we evaluate performance using data from the greater Boston road network (\cite{mbta2014,bertsimas2021data}). We also present additional experiments on synthetic data in Appendix \ref{sec:app-synthetic}
 
To accommodate the requirement that all available resources be scaled to $1$, costs are normalized to the $[0,1]$ range. Specifically, for each cost type $k$, a positive integer value $J_k$ is fixed, and all costs of this type are scaled proportionally so that the maximum cost of this type is $\frac{1}{J_k}$. Since one unit of each resource is available, a higher value of $J_k$ corresponds to higher availability of resource $k$. For simplicity, we used the same value $J_k=J$ for all three cost types.

The numerical experiments were performed using Python under the aforementioned parameters on a computing system equipped with an Intel(R) Core(TM) i9-14900K processor operating at 3.20 GHz and 32 GB of RAM. For each case, we obtain the LP solution, then perform 3,000 independent runs of Algorithm \ref{alg:low-cost-modified}, selecting the solution with the best objective value. Figure \ref{fig:comparison_results} presents the computational results, illustrating the performance of Algorithm \ref{alg:low-cost-modified} under different scenarios by varying the number of buses $M$ and the parameter $J$.

\begin{figure}
    \centering
    \includegraphics[width=\linewidth]{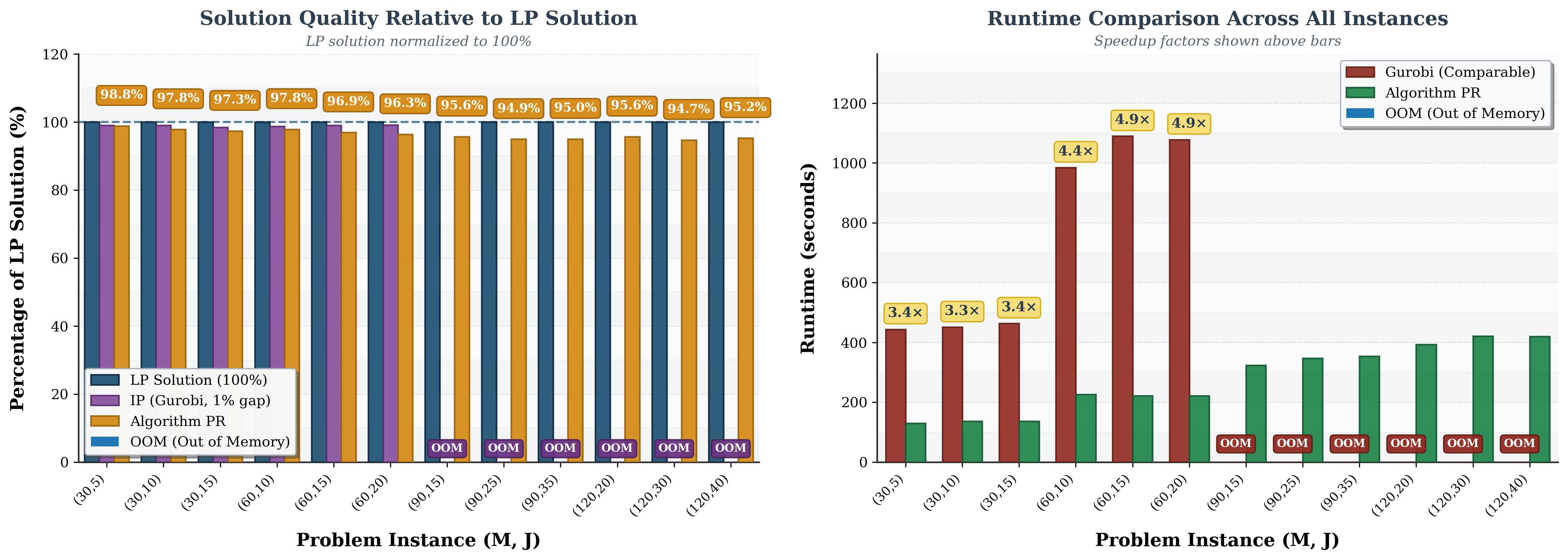}
    \caption{Computational Performance Comparisons.}
    \label{fig:comparison_results}
\end{figure}

The left panel of Figure \ref{fig:comparison_results} shows solution quality relative to the LP solution (normalized to 100\%). Algorithm \ref{alg:low-cost-modified} demonstrates remarkably consistent performance across all scenarios, achieving between 95\% and 98\% of the LP solution values. This consistency suggests that Algorithm \ref{alg:low-cost-modified} scales effectively with increasing problem size while maintaining near-optimal solution quality.
In contrast, when using Gurobi to directly solve the IP formulation, memory constraints become a critical limiting factor. For larger problems with $M \geq 90$, Gurobi fails to compute solutions entirely, encountering out-of-memory (OOM) errors.

The right panel of Figure \ref{fig:comparison_results} compares runtime performance. To ensure a fair comparison, we record the time at which Gurobi first reaches an intermediate solution slightly inferior to the output of Algorithm \ref{alg:low-cost-modified}. The speedup factors displayed above each bar pair demonstrate our algorithm's superior efficiency. It consistently achieves 3--5$\times$ faster computation while producing solutions of equal or better quality relative to this benchmark. For instance, in the case $(M = 60, J = 15)$, Algorithm \ref{alg:low-cost-modified} finds a superior solution in just 222 seconds, whereas Gurobi requires 1{,}090 seconds to reach an objective value that is still slightly lower. Moreover, Algorithm \ref{alg:low-cost-modified} scales robustly, with execution times remaining between 130 and 420 seconds, while Gurobi's runtime grows substantially until it eventually fails due to memory constraints.

\begin{figure}[htbp]
    \centering
    \begin{minipage}[b]{0.32\textwidth}
        \centering
        \includegraphics[width=1\linewidth]{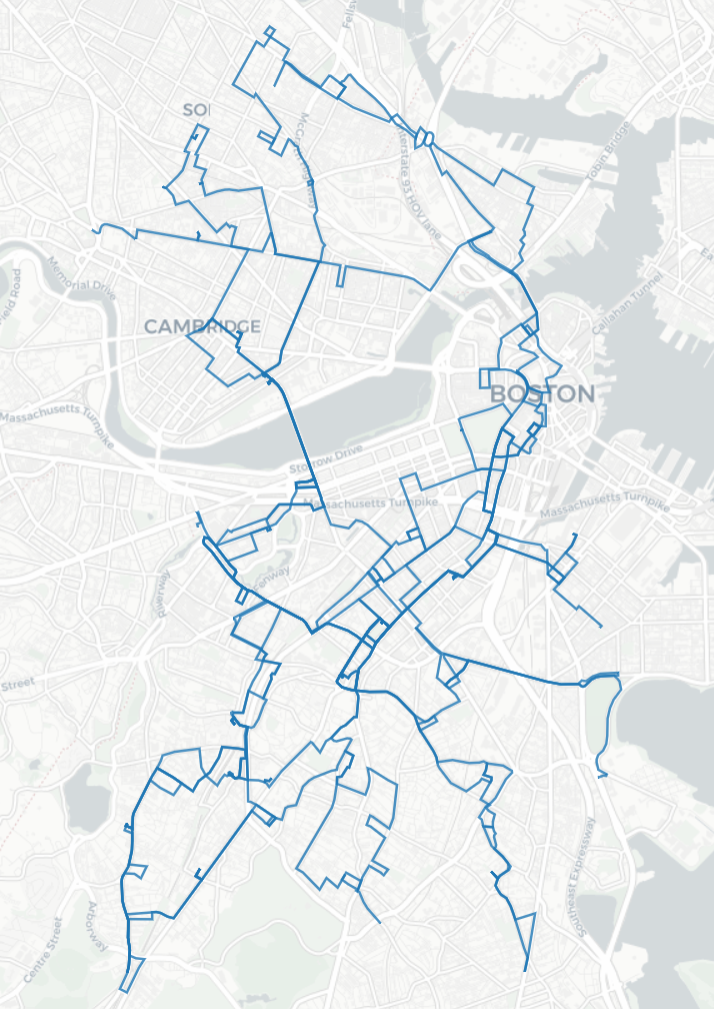}
        \par\vspace{-2pt} 
        \footnotesize{(a) Selected Lines ($M=90,~J=15$)}
        \label{fig:lines_15}
    \end{minipage}
    \hfill 
    \begin{minipage}[b]{0.32\textwidth}
        \centering
        \includegraphics[width=1\linewidth]{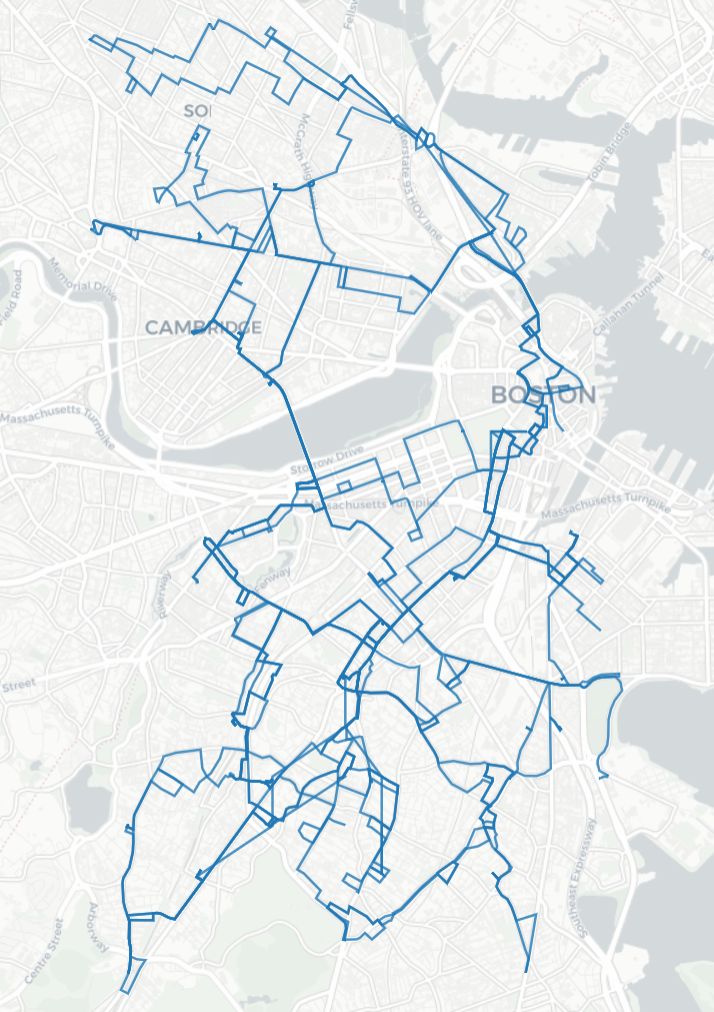}
        \par\vspace{-2pt} 
        \footnotesize{(b) Selected Lines ($M=90,~J=25$)}
        \label{fig:lines_25}
    \end{minipage}
    \hfill 
    \begin{minipage}[b]{0.32\textwidth}
        \centering
        \includegraphics[width=1\linewidth]{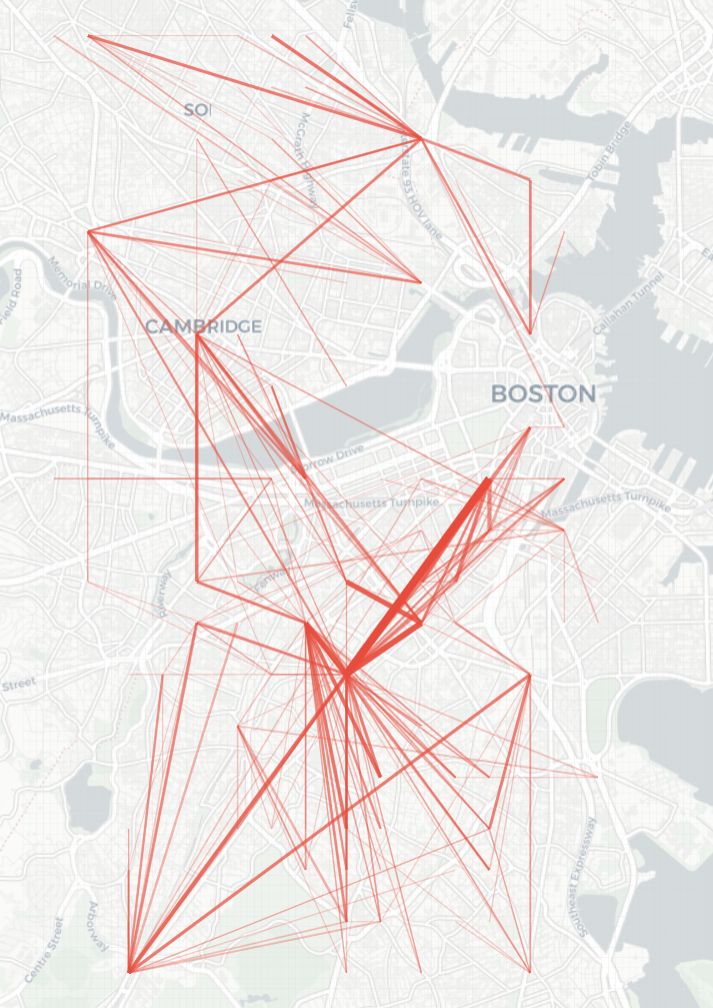}
        \par\vspace{-2pt} 
        \footnotesize{(c) Top-200 OD Flows}
        \label{fig:od_flows_vis}
    \end{minipage}

    \caption{Optimization results.  (a, b) Optimized network under resource constraints. (c)  Top 200 OD flows.}
    \label{fig:optimization_results}
\end{figure}
To illustrate the resulting line planning from our algorithm, we present visualizations in Figure~\ref{fig:optimization_results}. Panel~(c) depicts the top 200 OD flows by demand over the Boston-Cambridge metropolitan area, where red segments represent individual OD pairs. Notable demand concentrations appear along the north-south axis connecting Cambridge to downtown Boston, with significant flows radiating southward from the urban core.

Panels~(a) and (b) display the transit networks generated by Algorithm \ref{alg:low-cost-modified} for a fleet of $M=90$  buses under resource constraint parameters $J=15$ and $J=25$ respectively, with routes visualized using realistic road geometries via the OSRM routing API. A comparison with the OD flows in panel~(c) reveals strong spatial alignment. The optimized routes consistently cover high-demand corridors, particularly the Cambridge--Boston spine and the southern radial arteries where passenger flows are concentrated. This correspondence confirms that Algorithm \ref{alg:low-cost-modified} effectively captures the underlying demand structure.

\subsection{Ex-post Analysis of Passenger Preferences} 
\label{app:pref-analysis}

As discussed in Section~\ref{sec:model-setting}, our formulation accommodates both Route Assignment and Route Choice perspectives through the reward parameters $v_{b,l,(i,j)}$. Here we verify that the resulting assignments consistently place passengers on high-quality lines that closely match their individual preferences.

For each served passenger, we identify the set of active bus--line pairs capable of serving that passenger's OD pair in the realized solution and compute the \emph{objective gap} between the assigned option and the best available alternative as
$
\text{gap}_{(i,j)} = ({v^{\star}_{(i,j)} - v^{\text{assigned}}_{(i,j)}})/{v^{\star}_{(i,j)}} \times 100\%,
$
where $v^{\star}_{(i,j)}$ is the reward of the best available option. A gap of $0\%$ means the passenger is assigned to their most preferred service.

We perform this analysis on solutions generated by Algorithm~\ref{alg:low-cost-modified} for a heterogeneous fleet of $M = 120$ candidate buses with $J \in \{20, 30, 40\}$. To isolate the effect of reward-aware optimization, we also run an ablation test where the LP and rounding use uniform rewards (one unit per served passenger) while the post-analysis still evaluates gaps against the original reward. The ablation test achieves LP integrality gaps below 3\% in all cases, confirming that the rounding procedure remains effective even under uniform rewards. Figure~\ref{fig:pref-gap-cdf} presents the cumulative distribution of the objective gap for both settings.

\begin{figure}[ht]
\centering
\begin{minipage}[t]{0.48\textwidth}
    \centering
    \includegraphics[width=\linewidth]{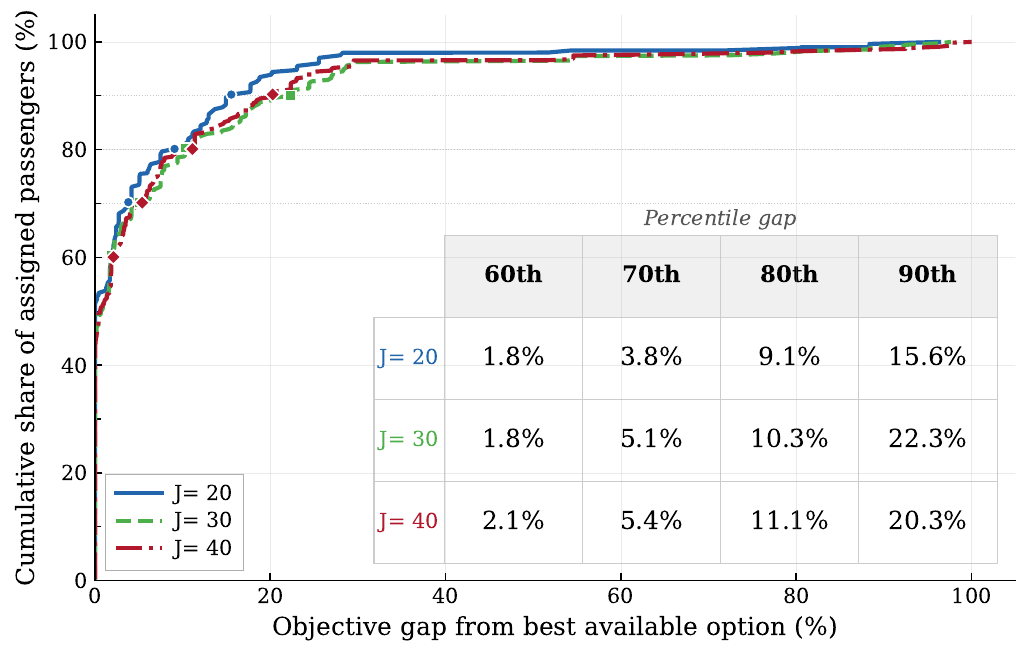}
\end{minipage}
\hfill
\begin{minipage}[t]{0.48\textwidth}
    \centering
    \includegraphics[width=\linewidth]{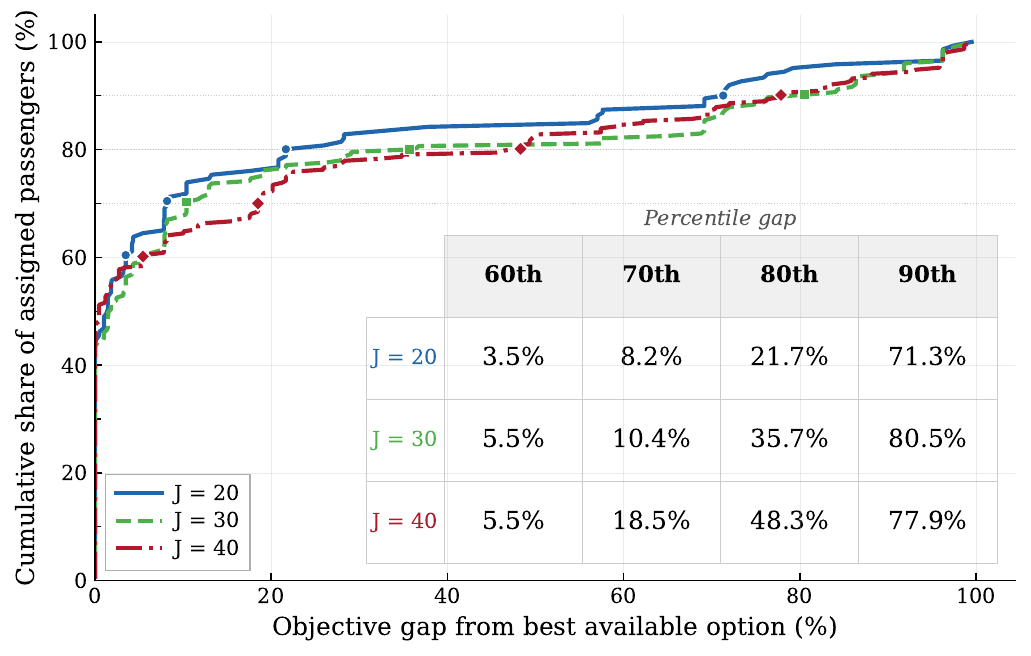}
\end{minipage}
\caption{Cumulative distribution of the objective gap from each passenger's best available option for $M = 120$ with $J \in \{20,30,40\}$. Left panel uses the original reward in optimization and right panel uses uniform reward (ablation).}
\label{fig:pref-gap-cdf}
\end{figure}
The results demonstrate that Algorithm~\ref{alg:low-cost-modified} with reward-aware
optimization assigns the vast majority of passengers to lines very close to their best
available option. In the left panel, roughly half of all served passengers are placed on
their best available line (0\% gap), and across all three settings of budget the
60th-percentile gap remains below $2.1\%$, the 70th-percentile gap stays within $5.4\%$,
and even the 80th-percentile gap ranges from only $9.1\%$ ($J{=}20$) to $11.1\%$
($J{=}40$). The three curves nearly overlap, indicating that reward-aware optimization
maintains tight preference alignment regardless of the amount of resources. The
ablation with uniform rewards (right panel) reveals a markedly different pattern. While a
comparable initial fraction of passengers still lands on their best option due to network
structure and capacity alone, the upper percentiles deteriorate sharply. The
70th-percentile gaps widen to $8.2\%$--$18.5\%$ (versus $3.8\%$--$5.4\%$ under
reward-aware optimization), the 80th-percentile gaps jump to $21.7\%$--$48.3\%$ (versus
$9.1\%$--$11.1\%$), and the 90th-percentile gaps reach $71.3\%$--$80.5\%$, roughly
$3{\times}$--$4{\times}$ their reward-aware counterparts. Moreover, the ablation curves
separate as $J$ increases, with the 80th-percentile gap growing from $21.7\%$ at $J{=}20$
to $48.3\%$ at $J{=}40$, suggesting that a reward-blind optimizer disperses passengers
across increasingly suboptimal alternatives when more resources are available. This contrast
confirms that the reward structure in our formulation plays a critical role in aligning
system-optimal assignments with individual passenger preferences. We conduct the same
analysis with $90$ candidate buses in Appendix~\ref{app:pref-analysis-90}, where the results exhibit
a similar pattern.
}

\section{Conclusion}\label{sec:conclusion}

This paper presents the \emph{Line Planning with Resource Constraints} (LPRC) framework, unifying heterogeneous fleets with multi-resource limits. Theoretically, we established a tight $1-1/e$ approximation guarantee for the base case and constant-factor bounds for general settings. We modified the algorithm for practical application, preserving a solid theoretical guarantee while ensuring computational scalability. In experiments, this approach consistently achieved 95\%--98\% of the LP bound, outperforming commercial solvers that failed due to memory constraints. Our results demonstrate that leveraging fleet heterogeneity improves service levels by over 20\% compared to homogeneous fleets, while multi-resource optimization is essential to prevent the operational infeasibility common in single-resource models.

A natural extension is to model stochastic origin--destination demand and maximize expected reward, with attention to reliability through chance constraints, distributionally robust formulations, and adaptive recourse as information arrives. Another critical direction is to incorporate fairness considerations, ensuring that the pursuit of aggregate efficiency and solid theoretical guarantee does not compromise equitable service access across diverse geographic or demographic groups. Finally, from a theoretical perspective, future work should aim to tighten the approximation ratios for the two extended cases, closing the gap between these general settings and the base case results.

\bibliography{sample}
\bibliographystyle{plainnat}

\newpage


%
%
%
\begin{APPENDICES}

\section{Addtional Numerical Experiments}\label{sec:addtional-numerical}

\subsection{Experiments on Synthetic Data}\label{sec:app-synthetic}
While the Boston-area experiments validate our approach on real-world data, synthetic instances allow us to evaluate algorithmic performance across a broader range of problem scales under controlled conditions.

Our synthetic instance generator places stations uniformly at random on a Euclidean grid and generates OD demand using a gravity model with a monocentric urban structure. Unlike the Boston-area experiments where passengers board and alight only at their exact origin and destination stations, the synthetic instances allow passengers to walk to nearby stations within a specified radius. This increases the number of feasible passenger-line assignments and creates denser constraint structures, making the optimization problem more challenging. The reward $v_{b,l,(i,j)}$ is set to an independent uniform random draw from $[0,1]$ for each feasible bus--line--OD combination (determined by the walking radius and line coverage), and zero otherwise.

We generate 30 independent instances with 600 stations, 600 OD pairs, 120 candidate buses across three capacity levels and two fuel types as in the Boston experiment. For each bus group, 1{,}000 candidate lines (6{,}000 total) are generated using a similar method as in Section~\ref{sec:computational-exp}. The normalization parameter $J = 50$ places these instances in a computationally demanding regime. The Gurobi MIP benchmark is allowed 30 minutes with a 1\% optimality gap. 

Table~\ref{tab:synthetic-results} summarizes the results. Our algorithm consistently achieves 96.3--96.5\% of the LP upper bound on average across all instance groups, with individual instances ranging from 95.7\% to 97.0\%. Remarkably, Gurobi fails to find any feasible solution within the time limit in 22 out of 30 instances, none of which suffer from the memory issues reported in Section \ref{sec:computation-advance}. In the 8 instances where Gurobi does find solutions, it outperforms our algorithm in one case (98.5\% vs.\ 96.3\%), while substantially underperforming in others, with ratios as low as 82.3\%. Our algorithm completes in an average of 9 minutes per instance. These findings confirm that our approach scales well to large-scale problem instances where commercial solvers struggle.

\begin{table}[htbp]
\centering
\caption{Synthetic Instance Results}
\label{tab:synthetic-results}
\scriptsize
\begin{tabular}{cccc|cccc|cccc}
\toprule
\multicolumn{4}{c|}{\textbf{Instances 1--10}} & \multicolumn{4}{c|}{\textbf{Instances 11--20}} & \multicolumn{4}{c}{\textbf{Instances 21--30}} \\
\midrule
\textbf{Inst} & \textbf{Alg/LP} & \textbf{Gurobi/LP} & \textbf{Time (s)} & \textbf{Inst} & \textbf{Alg/LP} & \textbf{Gurobi/LP} & \textbf{Time (s)} & \textbf{Inst} & \textbf{Alg/LP} & \textbf{Gurobi/LP} & \textbf{Time (s)} \\
\midrule
1  & 96.4\% & 0\%    & 486 & 11 & 96.6\% & 0\%    & 507 & 21 & 96.4\% & 0\%    & 570 \\
2  & 96.3\% & 0\%    & 446 & 12 & 97.0\% & 0\%    & 635 & 22 & 96.1\% & 0\%    & 564 \\
3  & 96.4\% & 0\%    & 439 & 13 & 96.3\% & 0\%    & 475 & 23 & 96.5\% & 0\%    & 631 \\
4  & 96.2\% & 0\%    & 590 & 14 & 96.3\% & 0\%    & 452 & 24 & 96.5\% & 0\%    & 659 \\
5  & 96.5\% & 82.7\% & 461 & 15 & 97.0\% & 88.8\% & 460 & 25 & 96.0\% & 0\%    & 460 \\
6  & 96.5\% & 82.3\% & 506 & 16 & 96.3\% & 95.9\% & 498 & 26 & 96.4\% & 0\%    & 611 \\
7  & 96.2\% & 96.0\% & 499 & 17 & 96.2\% & 0\%    & 455 & 27 & 97.0\% & 0\%    & 639 \\
8  & 96.3\% & 98.5\% & 248 & 18 & 96.4\% & 96.4\% & 469 & 28 & 96.6\% & 0\%    & 677 \\
9  & 96.3\% & 0\%    & 475 & 19 & 97.0\% & 0\%    & 541 & 29 & 96.5\% & 0\%    & 619 \\
10 & 95.7\% & 0\%    & 482 & 20 & 96.3\% & 91.4\% & 568 & 30 & 96.3\% & 0\%    & 549 \\
\midrule
\textbf{Avg} & \textbf{96.3\%} & \textbf{--} & \textbf{463} & \textbf{Avg} & \textbf{96.5\%} & \textbf{--} & \textbf{506} & \textbf{Avg} & \textbf{96.4\%} & \textbf{--} & \textbf{598} \\
\bottomrule
\end{tabular}
\end{table}

\subsection{Additional Experiments to Section \ref{sec:heter}}\label{sec:app-het}
We additionally evaluate a 60-bus fleet under proportionally equivalent budget scenarios, where small, medium, and large budgets permit approximately $20$, $30$, and $40$ heterogeneous buses to be deployed, respectively. The results are reported in Figure \ref{fig:het_adv_60bus}, which show the same pattern as in the 30-bus case, with heterogeneous fleets delivering a significant advantage.

\begin{figure}[htbp]
    \centering
    \includegraphics[width=\linewidth]{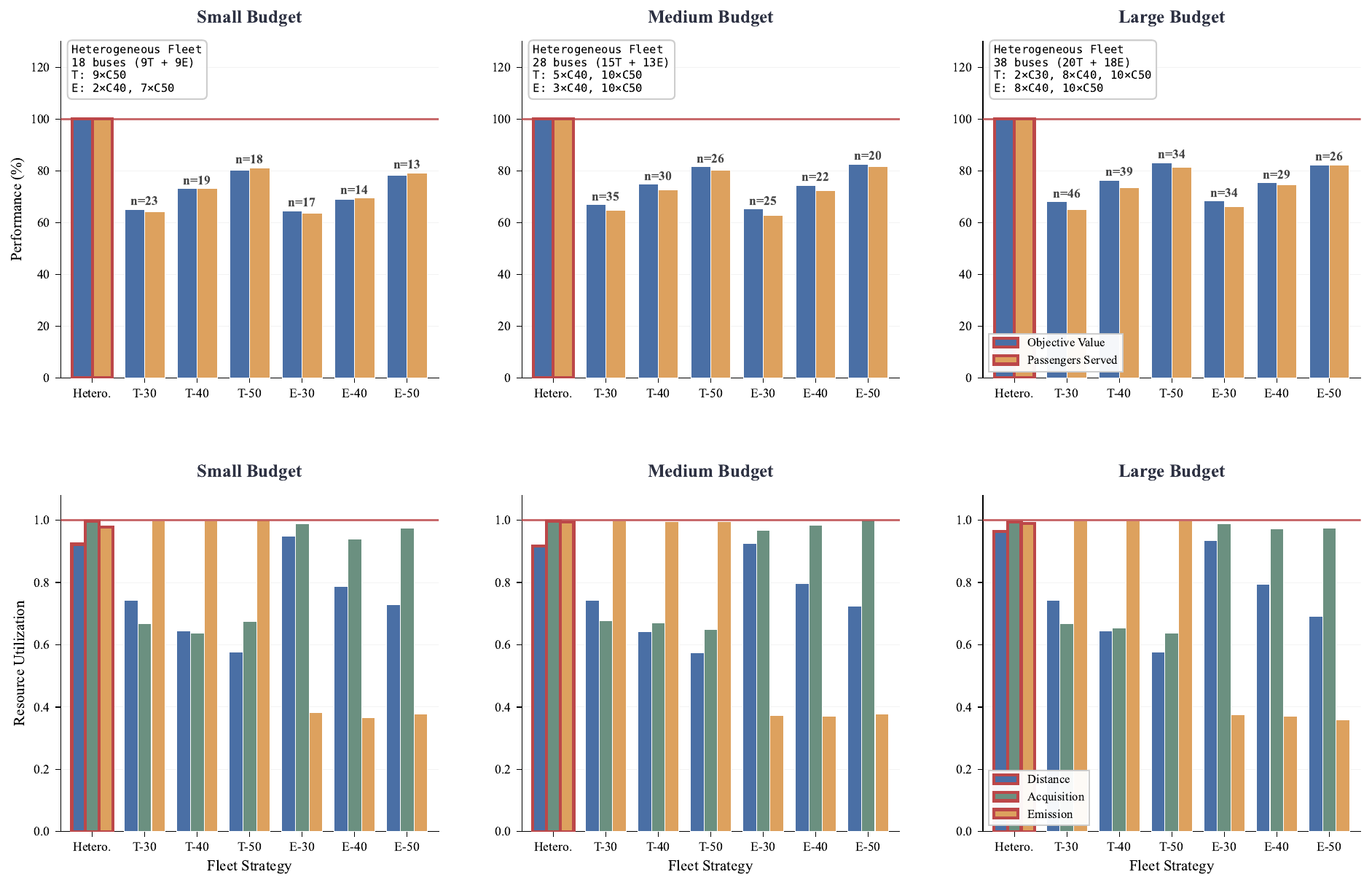}
    \caption{Performance and resource utilization comparison of heterogeneous versus homogeneous 60-bus fleets across budget scenarios.}
    \label{fig:het_adv_60bus}
\end{figure}

\subsection{Additional Experiments to Section \ref{sec:multi-resource-advantage}}\label{sec:app-single-resource}
We additionally evaluate a 60-bus fleet under the similar setup as the 30-bus experiment. The results are reported in Figure~\ref{fig:strategy_comparison_60}, which show the same pattern as in the 30-bus case.

\begin{figure}[h]
    \centering
    \includegraphics[width=\linewidth]{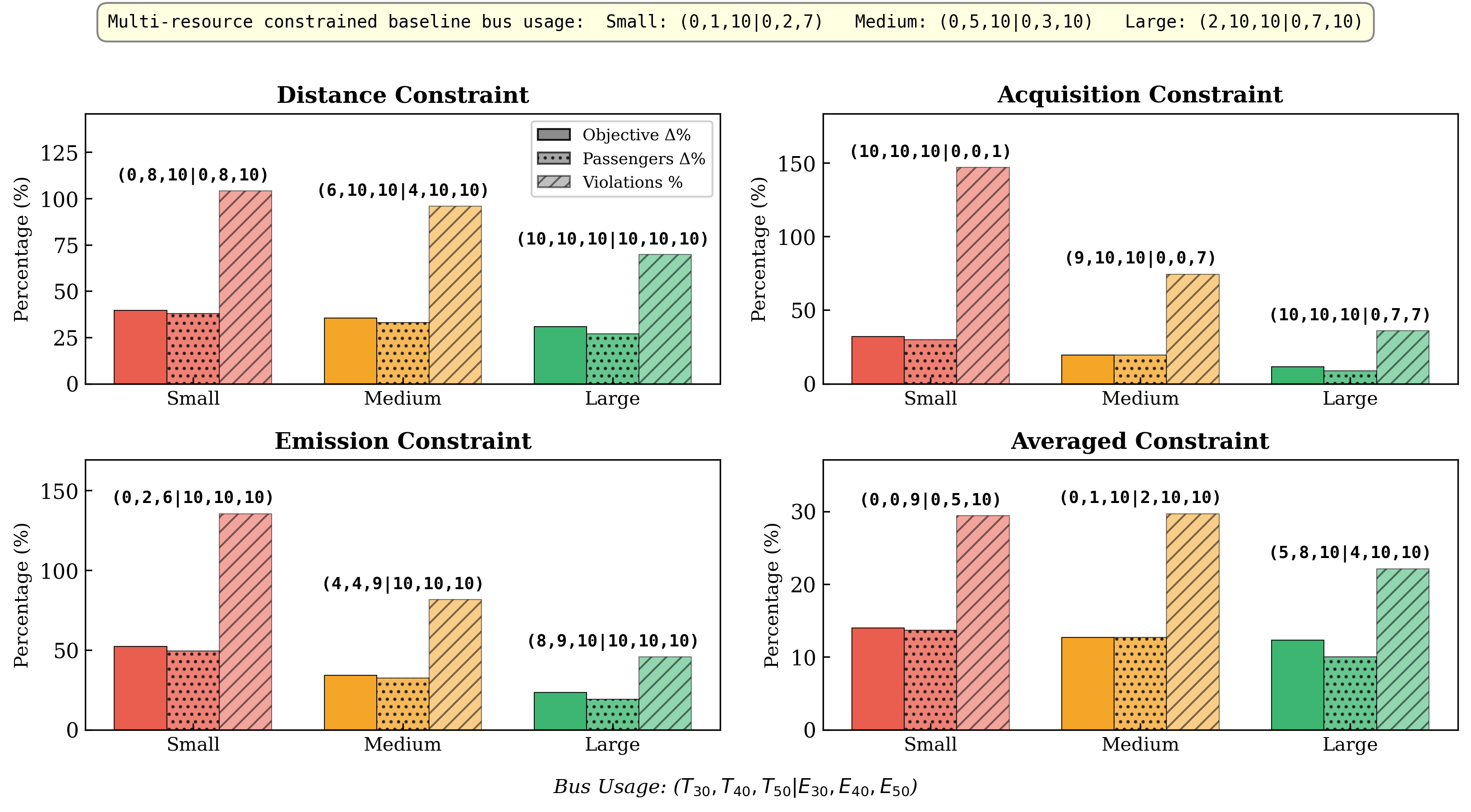}
    \caption{Performance comparison of single-resource optimization strategies across budget levels (60-bus).}
    \label{fig:strategy_comparison_60}
\end{figure}

\subsection{Additional Experiments to Section \ref{app:pref-analysis}}\label{app:pref-analysis-90}
Figure~\ref{fig:pref-gap-cdf-90} replicates the preference gap analysis of
Section~\ref{app:pref-analysis} for a smaller fleet of $M = 90$ candidate buses with
$J \in \{15, 25, 35\}$. The results confirm the same pattern observed for $M = 120$.
Under reward-aware optimization (left panel), the 60th-percentile gap remains below
$2.2\%$ and the 80th-percentile gap stays within $13.7\%$ across all settings of $J$.
The uniform-reward ablation (right panel) again shows substantial degradation in the
upper percentiles, with 80th-percentile gaps rising to $8.2\%$--$50.2\%$ (versus
$4.2\%$--$13.7\%$) and 90th-percentile gaps reaching $37.8\%$--$83.4\%$ (versus
$9.6\%$--$20.5\%$). These findings reinforce that the reward structure is essential for
aligning system-optimal assignments with passenger preferences across different fleet
sizes.

\begin{figure}[ht]
\centering
\begin{minipage}[t]{0.48\textwidth}
    \centering
    \includegraphics[width=\linewidth]{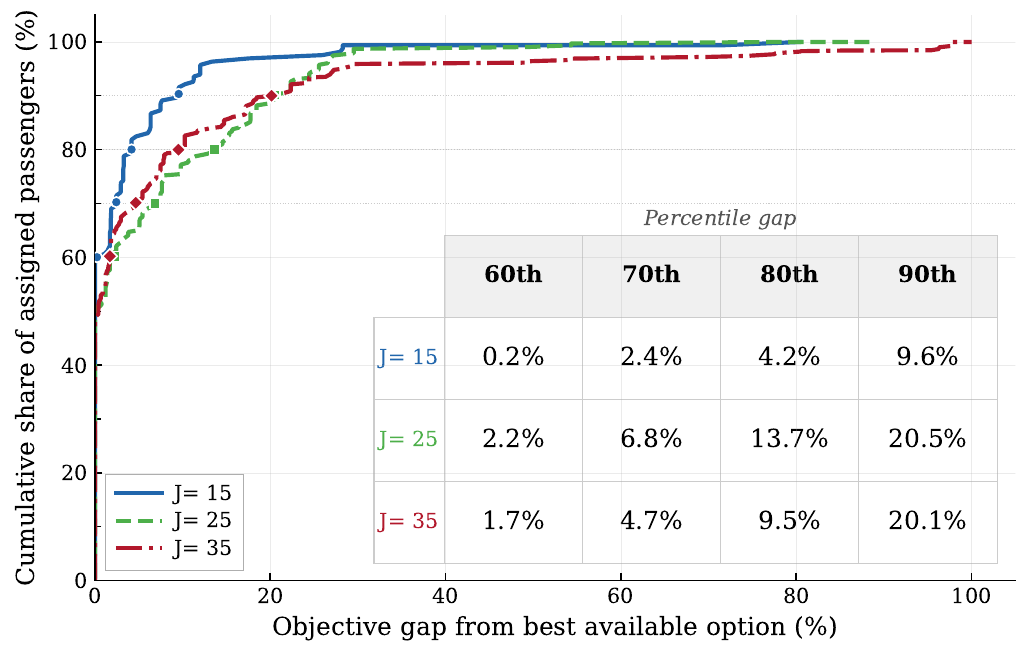}
\end{minipage}
\hfill
\begin{minipage}[t]{0.48\textwidth}
    \centering
    \includegraphics[width=\linewidth]{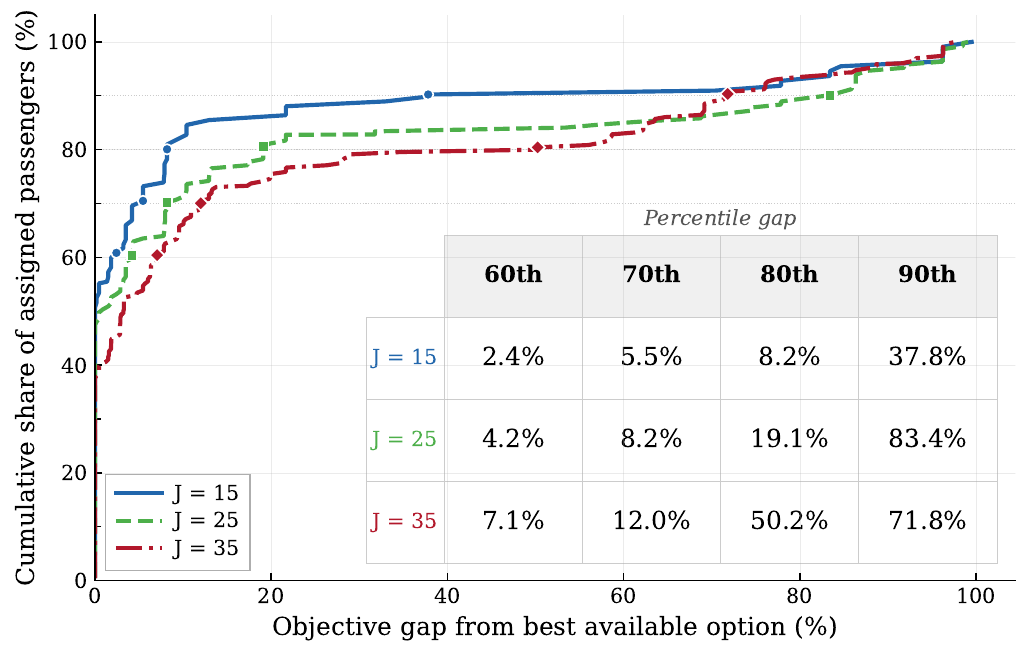}
\end{minipage}
\caption{Cumulative distribution of the objective gap from each passenger's best
available option for $M = 90$ with $J \in \{15, 25, 35\}$. Left panel uses the original
reward in optimization and right panel uses uniform reward (ablation).}
\label{fig:pref-gap-cdf-90}
\end{figure}

\section{Proof of Lemma \ref{lem:LP-relaxation-solvable}}\label{app-sec:lem:LP-solvable}

\begin{proof}
The dual problem of the LP relaxation (\ref{LP-relaxation}) can be written as: 
\begin{subequations}\label{dual-for-rounding-2}
\begin{align}
    \min_{q,u,w}\quad &\sum_{b\in [M]}q_b + \sum_{(i,j)\in D}d_{(i,j)}\cdot w_{(i,j)} + \sum_{k\in [K]} u_k\\
    \text{s.t.}\quad &q_b+\sum_{(i,j)\in D^{b,l}}\theta_{b,l,(i,j)}^{(t)}w_{(i,j)}  + \sum_{k\in [K]}c_{b,l}^{(k)}\cdot u_k\geq \sum_{(i,j)\in D^{b,l}}v_{b,l,(i,j)}\theta_{b,l,(i,j)}^{(t)},~\forall b,~ l,~t\label{eq:dual-ellipsoid}\\
    & w_{(i,j)}\geq 0 \ \text{and}\ u_k\geq 0.
\end{align}
\end{subequations}
It has an exponential number of constraints and a polynomial number of variables. Given specific values for the dual variables $q^*$, $u^*$ and $w^*$, we can check feasibility considering separately each $b\in [M]$ and $l\in \mathcal{L}_b$. That is, we need to check whether a $\theta_{b,l}^{(t)}$ exists for some $t\in [n_{b,l}]$ such that $\sum_{(i,j)\in D^{b,l}}\left(v_{b,l,(i,j)}-w^*_{(i,j)}\right)\theta_{b,l,(i,j)}^{(t)}$ exceeds $q^*_b+\sum_{k\in [K]}c_{b,l}^{(k)}\cdot u^*_k$. To do this, for each $b\in [M]$ and $l\in \mathcal{L}_b$, it is sufficient to solve:
 \begin{align*}
\max_{\theta}\quad&\sum_{(i,j)\in D^{b,l}}\left(v_{b,l,(i,j)}-w^*_{(i,j)}\right)\theta_{b,l,(i,j)}\\
     \text{s.t.}\quad &\theta_{b,l}\in \mathcal{P}(b,l)\cap \mathbb{Z}^{|D^{b,l}|}.
 \end{align*}

For each $b$ and $l$, the $0$–$1$ coefficient matrix formed by the constraints $\sum_{(i,j) \in D^{b,l}(e)} \xi_{b,l,(i,j)} \leq C_b$, $e\in l$ of $\mathcal{P}(b,l)$ has columns with consecutive $1$s. This is because the arcs traversed by any OD pair on line $l$ are always consecutive. Consequently, this matrix is totally unimodular~(\cite{fulkerson1965incidence}). Since $C_b \in \mathbb{Z}$ and $d_{(i,j)} \in \mathbb{Z}$ for each $(i,j) \in D$, the extreme points of $\mathcal{P}(b,l)$ are integral. Therefore, we only need to solve the following polynomial-size LP:
\begin{subequations}\label{eq:LP-subproblem}
    \begin{align}
 \max_{\theta}\quad&\sum_{(i,j)\in D^{b,l}}\left(v_{b,l,(i,j)}-w^*_{(i,j)}\right)\theta_{b,l,(i,j)}\\
     \text{s.t.}\quad &\theta_{b,l}\in \mathcal{P}(b,l).
 \end{align}
\end{subequations}

With the application of the ellipsoid method (using the above feasibility oracle) and standard reconstruction from the generated constraints/columns, we obtain in polynomial time an exact optimal solution of the dual. Using the ellipsoid method, the primal LP relaxation~\eqref{LP-relaxation} can be solved exactly in time polynomial in the input size (see, e.g., \cite[Section~4.3]{williamson2011design}).
\end{proof}

\bigskip

\section{Tightness of the $\left(1-\tfrac{1}{e}\right)$ Approximation for LPRC without Line Costs}\label{app-sec:thm:hardness}

In this section, we demonstrate that the approximation ratio for LPRC without costs cannot exceed $\left(1-\frac{1}{e}\right)$, and thus that Algorithm \ref{alg:simple-alg} achieves the optimal approximation ratio. We use a known hardness result for the max $k$-cover problem. We show that any instance of the max $k$-cover problem can be transformed into an instance of LPRC, so the hardness result extends to LPRC.

\begin{definition}\label{def:max-k}
Given a set $G = \{p_1, \dots, p_n\}$, and $F = \{G_1, G_2, \dots, G_L\}$, which is a collection of subsets of $G$, and an integer $k>0$, max $k$-cover is the problem of selecting $k$ sets from $F$ such that their union contains as many elements as possible.
\end{definition}

The following hardness result has been proven in \cite{feige1998threshold}. 

\begin{theorem} [\cite{feige1998threshold}]
For any constant $\epsilon > 0$,  unless $P = NP$, max $k$-cover problem cannot be approximated in polynomial time with an approximation ratio of $(1 - \frac{1}{e} + \epsilon)$.
\end{theorem}

Building on this result, we show that LPRC is at least as hard as the max $k$-cover problem, establishing a similar inapproximability bound. 
  
\begin{theorem}\label{thm:hardness}
    Unless \( P = NP \), the LPRC problem cannot be approximated in polynomial time within an approximation ratio of \(1 - \frac{1}{e} + \epsilon\) for any constant \(\epsilon > 0\), even when all buses have unit capacity and all OD pairs have unit demand.
\end{theorem}

\begin{proof}
Given an instance $(G = \{p_1, \dots, p_n\},~  F = \{G_1, G_2, \dots, G_L\},~k)$      of the max $k$-cover problem, we construct an equivalent instance of LPRC as follows. This instance will have $k$ buses, each with unit capacity, and a set of $L$ candidate lines.

The underlying network will have $2n$ vertices  $o_{i,1}, o_{i,2}$ for ${i\in [n]}$, and a set $D$  of $n$ OD pairs  $(o_{i,1}, o_{i,2})$  each having unit demand. The OD pair $(o_{i,1}, o_{i,2})$ corresponds to the element $p_i$ and will be identified with this element.   For any  set $G_j := \{p_{j_1}, \dots , p_{j_{|G_j|}}\}$ from $F$ with $j_1 \leq j_2 \leq \dots \leq j_{|G_j|}$, we create a candidate line $j$ that can only serve the OD pairs from $G_j$; see the illustrative example in Figure \ref{fig:hardness-line}. The line includes arcs from $o_{j_i,1}$ to $o_{j_i,2}$ for each $i\in\left[|G_j|\right]$, and arcs from $o_{j_i,2}$ and $o_{j_{i+1},1}$ for each $i\in \left[|G_j| -1\right]$, thereby establishing a directed path. Note that line $j$ can only serve the OD pairs from $G_j$ since for any OD pair $p_{i'}$ not in $G_j$, no path exists from $o_{i',1}$ to $o_{i',2}$ on line $j$. We  assign a reward of $1$ for each unit of demand served by any bus on any line. The illustrative example in Figure \ref{fig:hardness-line} shows a graph with $18$ vertices denoted by $\left\{o_{i,1},o_{i,2}\right\}_{i\in [9]}$. There are $9$ OD pairs represented as $\left\{p_1,p_2,\ldots, p_9\right\}$, where each pair $p_i=(o_{i,1}, o_{i,2})$ corresponds to travel from $o_{i,1}$ to $o_{i,2}$. Given the set $G_1:=\left\{p_1,p_3,p_9\right\}$, we construct line $1$ by sequentially connecting vertices $o_{1,1}$, $o_{1,2}$, $o_{3,1}$, $o_{3,2}$, $o_{9,1}$, and $o_{9,2}$. Note that this line can only serve OD pairs $p_1$, $p_3$, and $p_9$. 

The equivalence of this LPRC instance and the given max $k$-cover instance is straightforward. 
\end{proof}
\begin{figure}[h]
    \centering
    \includegraphics[width = 10cm]{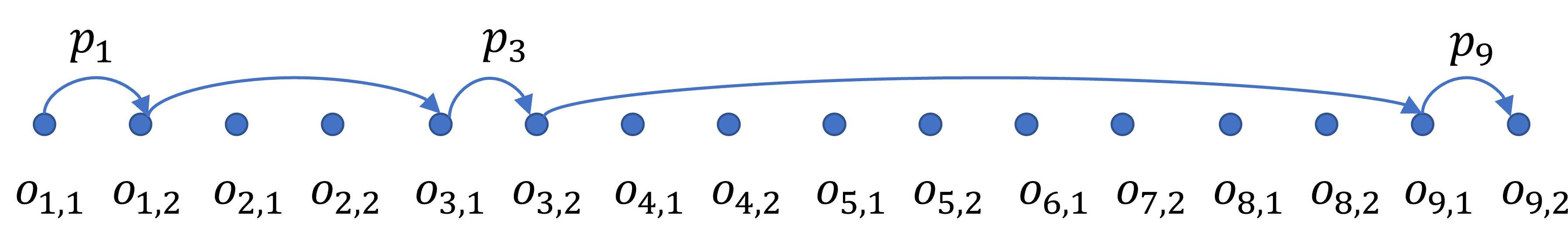}
    \caption{An illustration for the proof of Theorem \ref{thm:hardness} }
    \label{fig:hardness-line}
\end{figure}

\bigskip

\section{LPRC with Small Line Costs}\label{sec:low-cost}

Suppose a constant $\eta \in (0, \frac 12)$ is given. In this section, we introduce a randomized algorithm LC (which will also be called Algorithm LC($\eta$) to emphasize parameter $\eta$) tailored for LPRC instances where line costs are small. This will serve as a critical starting point for the theoretical analysis of the algorithms in Section \ref{sec:practical},  Appendix \ref{sec:approx-LPRC}, and \ref{sec:tolerance}. We formalize this condition with the following assumption. 

\begin{assumption}\label{assumption:small-costs}
 We assume that
$
\max_{b\in [M],l\in \mathcal{L}_b,k\in [K]}c_{b,l}^{(k)} \leq  \frac{\eta^3}{256K^3}.
$
\end{assumption}

\noindent {\bf Intuitive Overview:} 
We use a randomized rounding mechanism similar to Algorithm \ref{alg:simple-alg} but with somewhat smaller selection probabilities, introducing some additional probability that some buses will not be utilized. With that, for each resource, the expected resource usage will be somewhat smaller than the limit; then, since costs are small, the probability of exceeding the resource limit will be small. If the randomized solution exceeds the limit for some resource, we discard the solution (do not utilize the buses at all); since the probability of exceeding the resource limit is small, this does not affect significantly the expected reward.

\begin{namedalg}{LC}\label{alg:low-costs-feasible}

~

   \noindent {\bf Input:}
        LP relaxation (\ref{LP-relaxation}) and an optimal solution $\left\{\widehat{\lambda}_{b,l}^{(t)}: b\in [M], l\in \mathcal{L}_b, t\in [\widehat{n}_{b,l}]\right\}$ as in Assumption~\ref{assumption:hat-n}.

\noindent {\bf Output: }A (randomized) solution to (\ref{eq:ILP-original}).

\noindent {\bf Steps: }
\begin{enumerate}

    \item Let $q=\frac{4}{\eta}$ and $\epsilon = \frac{1}{q\cdot K}$. Observe that $\epsilon < \frac{\eta}{2}$.
For each $b\in [M]$, sample a tuple $(l,t)$ with the selection probability given by $(1-\epsilon)\widehat{\lambda}_{b,l}^{(t)}$. Denote the sampled tuple for each $b\in [M]$ as $(b,l_b,t_b)$. Note that the total of the selection probabilities for bus $b$ is less than 1. If no tuple is selected, the bus is not used (assigned to a ``dummy" line with zero costs and rewards).\label{alg-step:LC-sample}
    \item Same as Step \ref{simple-alg-step:assign} in Algorithm \ref{alg:simple-alg}. \label{alg-step:LC-rounding}
    \item If the solution meets the resource constraints, it becomes the final output. Otherwise, no lines are assigned to buses, and the reward generated is $0$. \hfill\label{alg-step:LC-discard}
\end{enumerate} 

\end{namedalg}

\noindent Recall that $\Gamma: = \sum_{b\in [M]}\sum_{l\in \mathcal{L}_b}\sum_{t \in [n_{b,l}]}\sum_{(i,j)\in D^{b,l}}v_{b,l,(i,j)}{\theta}_{b,l,(i,j)}^{(t)}\widehat{\lambda}_{b,l}^{(t)}$ is the optimal objective value of  the LP relaxation (\ref{LP-relaxation}). We first analyze the solution provided by Steps \ref{alg-step:LC-sample} and \ref{alg-step:LC-rounding} of Algorithm \ref{alg:low-costs-feasible}.

\begin{lemma}\label{lem:solution-alg-low-cost}
Let $\delta = \frac{\eta^3}{256K^3}$. The randomized solution returned by Steps \ref{alg-step:LC-sample} and \ref{alg-step:LC-rounding} of Algorithm \ref{alg:low-costs-feasible} has the following properties:
\begin{enumerate}[(1)]
    \item The expected reward is at least $\left(1-\frac{1}{e}-\epsilon\right)\Gamma$.
    \item Consider the event that the solution uses more than \( \psi \) units of at least one resource, where \( \psi \geq 1 \). The probability of this event occurring is bounded from above by \( K \cdot e^{-(\psi \epsilon^2)/(4 \delta)} \).
    \item If the solution uses no more than \(\psi+1\) units of each resource, the resulting reward is bounded from above by \((\psi+1) \cdot \Gamma\).

\end{enumerate}
\end{lemma}

\begin{proof}
   Property (1) can be established by employing an analysis similar  to that used in the proof of Theorem \ref{thm:main}. The presence of the term \( \epsilon \) in the approximation ratio arises from the adjustment of the sampling probabilities by a scaling factor of \( (1 - \epsilon) \).

    Property (2) will be proven by applying Chernoff bounds (\cite{motwani1995randomized}). Particularly, let $Y_{b,l}$ be the random indicator variable denoting the event that bus $b$ is assigned to line $l$. Note that if $Y_{b,l}=0$ for all $l\in \mathcal{L}_b$, then it means that bus $b$ is not utilized in the randomized solution. For a given $k\in [K]$, we have that 
    \begin{align*}
        P\left(\sum_{b\in [M]}\sum_{l\in \mathcal{L}_b}Y_{b,l}\cdot c_{b,l}^{(k)}\geq \psi(1+\epsilon)(1-\epsilon)\right)
        &=P\left(\sum_{b\in [M]}\sum_{l\in \mathcal{L}_b}Y_{b,l}\cdot \frac{c_{b,l}^{(k)}}{\delta}\geq \psi(1+\epsilon)\frac{(1-\epsilon)}{\delta}\right)\leq e^{-(\psi\epsilon^2)/(4\delta)}
    \end{align*}
    The last inequality is due to  $\sum_{l\in \mathcal{L}_b}Y_{b,l}\cdot \frac{c_{b,l}^{(k)}}{\delta}\in [0,1]$ according to Assumption \ref{assumption:small-costs} , and $\sum_{b\in [M]}\sum_{l\in \mathcal{L}_b}\E \left[Y_{b,l}\right]\cdot c_{b,l}^{(k)}\leq (1-\epsilon)$, $\epsilon<\frac{1}{8}$ and Chernoff bounds. Now by taking a union bound over all $K$ resources, we can finish the proof.

    Property (3) can be verified using concavity of the optimal objective value of a maximization linear programming problem with respect to right-hand sides of resource limit constraints.
\end{proof}





    








\noindent 

The following theorem establishes a guarantee for the expected approximation ratio of the randomized solution generated by Algorithm \ref{alg:low-costs-feasible}.

\begin{theorem}\label{thm:alg-small-costs}
The expected reward of the solution obtained by Algorithm \ref{alg:low-costs-feasible} is at least
 \(  \left(1 - \frac{1}{e} - \eta\right) \cdot \Gamma \) under Assumption \ref{assumption:small-costs}.

\end{theorem}

\begin{proof}
Let $q=\frac{4}{\eta}$ and $\delta = \frac{1}{4q^3K^3} = \frac{\eta^3}{256K^3}$. Let \( \mathcal{B}_\psi \) be the event that  Steps \ref{alg-step:LC-sample} and \ref{alg-step:LC-rounding} of Algorithm \ref{alg:low-costs-feasible} return a randomized solution which utilizes at least \( \psi \) units for one or more resources, but does not exceed \( \psi+1 \) units for any resource. We define $P_\psi := P(\mathcal{B}_\psi)$ and
\[
E_\psi := \E\left[\text{Objective value of the randomized solution generated by  Step \ref{alg-step:LC-sample} and \ref{alg-step:LC-rounding} of Algorithm \ref{alg:low-costs-feasible}} ~\Big\vert~ \mathcal{B}_\psi\right].\]

\noindent With this notation, \( E_0\cdot P_0 \) is the expected objective value of the solution yielded by Algorithm \ref{alg:low-costs-feasible}.

By definitions and Lemma \ref{lem:solution-alg-low-cost}, we have that
\begin{align*}
    &\left(1-\frac{1}{e}-\epsilon\right)\Gamma
    \leq E_0\cdot P_0 + \sum_{\psi= 1}^{\infty}E_\psi\cdot P_\psi\\
    \leq &E_0\cdot P_0 + \sum_{\psi= 1}^{\infty}(\psi+1)\Gamma\cdot K\cdot e^{-\psi\epsilon^2/4\delta}=E_0\cdot P_0 + \Gamma\cdot \sum_{\psi= 1}^{\infty}(\psi+1)\cdot K\cdot e^{-\psi Kq}\\
    \leq &E_0\cdot P_0 + 2\Gamma\cdot \sum_{\psi= 1}^{\infty}\psi\cdot K\cdot e^{-\psi Kq}\leq E_0\cdot P_0 + 2\Gamma\cdot \sum_{t= 1}^{\infty}t\cdot e^{-t \cdot q}\\
    \leq&   E_0\cdot P_0 + \frac{2}{q}\cdot \Gamma = E_0\cdot P_0 + \frac{\eta}{2}\cdot \Gamma.
\end{align*}
The last inequality is due to the fact that $q=\frac{4}{\eta}> 4$, and \begin{align*}
    \sum_{t= 1}^{\infty}t\cdot e^{-t\cdot q} = e^{-q} + \sum_{t= 2}^{\infty}t\cdot e^{-t\cdot q}\leq  e^{-q} + \int_{1}^\infty (x+1)e^{-x\cdot q}dx= \frac{e^{-q}}{q}\cdot (2+\frac{1}{q} + q)\leq \frac{1}{q}.
\end{align*}
Since $\epsilon \leq \frac{\eta}{2}$, we have  \( E_0\cdot P_0 \geq \left(1 - \frac{1}{e} - \eta\right) \cdot \Gamma \). 
\end{proof}

\bigskip

\section{LPRC with General Line Costs}\label{sec:approx-LPRC}

In this section, we introduce an approximation algorithm  for general LPRC that achieves an expected approximation ratio of \( \left(\frac{1}{2} - \frac{1}{2e} - \eta'\right) \) for any given constant $\eta'\in (0,1/4)$. 

Before detailing the algorithm, we define some specific sets and functions for formalization.

\begin{definition}
Consider an instance \eqref{eq:ILP-original} of LPRC. Let \( \mathcal{A} \) be the set of all binary vectors \( \omega = (\omega_{b,l})_{b \in [M],l \in \mathcal{L}_b }\) such that the following conditions are met:
    $\sum_{b \in [M], l \in \mathcal{L}_b} c_{b,l}^{(k)} \cdot \omega_{b,l} \leq 1$, for all $k \in [K]$, and
    $\sum_{l \in \mathcal{L}_b} \omega_{b,l} = 1$, for all $b \in [M]$. Thus, each element in \( \mathcal{A} \) specifies an assignment of candidate lines to buses that complies with the resource constraints, and \( \mathcal{A} \) includes all such assignments.

We introduce a function \( f_{L}: \mathcal{A} \rightarrow \mathbb{Q}_+ \). For \( \omega \in \mathcal{A} \), \( f_{L}(\omega) \) denotes the optimal objective value for the LP relaxation \eqref{LP-relaxation} subject to the additional constraints that \( \lambda_{b,l}^{(t)} = 0 \) whenever \( \omega_{b,l} = 0 \). We use \( \text{LP}(\omega) \) to denote this restricted LP formulation. \( \text{LP}(\omega) \) is the LP relaxation of LPRC with the candidate line assignments fixed as defined by \( \omega \).
\end{definition}

\begin{definition}
   Given a constant $\delta\in (0,1)$, for each $b\in [M]$, define the set \( \mathcal{L}_b^\delta \) as follows:
    \[
    \mathcal{L}_b^{\delta} := \left\{ l\in \mathcal{L}_b : c_{b,l}^{(k)} \leq \delta \text{ for all } k \in [K] \right\}.
    \]
    This set contains the candidate lines of bus $b$ for which all associated costs are not greater than \( \delta \).  Define the set $ \mathcal{A}_\delta:= \left\{ \omega \in \mathcal{A} : \omega_{b,l} = 0 \text{ if } l \in \mathcal{L}_b^\delta \right\}.
    $
    This set includes all feasible line assignments limited to lines not in $\mathcal{L}_b^\delta$.


    Let \( \text{LP}_\delta \) be the LP relaxation (\ref{LP-relaxation}) subject to extra constraints  \( \lambda_{b,l}^{(t)} = 0 \) whenever \( l\notin \mathcal{L}_b^\delta \), and let $\Gamma_\delta$ be its optimal objective value.  \( \text{LP}_\delta \) represents the LP relaxation of LPRC including only the lines with all costs not greater than \( \delta \).
\end{definition}

\noindent {\bf Algorithm Overview:} The core concept of the algorithm is to initially identify a candidate line assignment $\omega^*$ in \( \mathcal{A}_\delta \) that yields the best LP-relaxation objective value \( f_{L}(\omega^*) \), and then compare \( f_{L}(\omega^*) \) with the objective value \( \Gamma_\delta \) of \( \text{LP}_\delta \). Then, we apply the appropriate randomized rounding algorithm to the better (in terms of the optimal objective value) of these two LP relaxations, to obtain a final solution for LPRC. (Alternatively, we can apply the appropriate randomized rounding to the obtained optimal solution of  \( \text{LP}_\delta \) and to the obtained optimal solution of \( \text{LP}(\omega^*) \), and choose the best of the two obtained solutions.)

\begin{namedalg}{C}\label{alg:approx-all-costs}
    
~

   \noindent {\bf Input:}
        An instance of LPRC (\ref{eq:ILP-original}), and $\eta'\in \left(0,\frac{1}{4}\right)$. Set $\eta=2\eta'$.

\noindent {\bf Output: }A (randomized) solution to (\ref{eq:ILP-original}).

\noindent {\bf Steps: }
\begin{enumerate}
    \item Define \( q = \frac{4}{\eta} \) and \( \delta = \frac{1}{4q^3K^3} = \frac{\eta^3}{256K^3} \). Solve $LP(\omega)$ for all $\omega \in \mathcal{A}_{\delta}$. 
     Let $\omega^* \in \argmax_{\omega\in \mathcal{A}_\delta}f_L(\mathbf{\omega}).$ 
     \item If $\Gamma_\delta\geq f_L(\mathbf{\omega^*})$, then apply Algorithm \ref{alg:low-costs-feasible}($\eta$) with $\eta$ as chosen above to derive a solution to (\ref{eq:ILP-original}) based on $\text{LP}_\delta$.
    
    \item Otherwise,  apply Algorithm \ref{alg:simple-alg} to derive a solution to (\ref{eq:ILP-original}) based on $\text{LP}(\omega^*)$.\hfill
\end{enumerate}

\end{namedalg}

\noindent The following result guarantees that Algorithm \ref{alg:approx-all-costs} can be completed in polynomial time. Note that according to the assumptions made, $\frac{K}{ \delta}$ is a constant.

\begin{lemma}
    Let \( L := \max_{b \in [M]} |\mathcal{L}_b| \). Then, we can bound \( |\mathcal{A}_\delta| \) as follows:
    $|\mathcal{A}_\delta| \leq \frac{K}{ \delta}\cdot (M \cdot L)^{\frac{K}{ \delta}}.$
\end{lemma}

\begin{proof}
    By  the pigeonhole principle, at most \( \frac{K}{\delta} \) lines with at least one type of cost larger than $\delta$ can be utilized while satisfying all resource constraints. Since each bus can choose from a maximum of \( L \) candidate lines, the bound on \( |\mathcal{A}_\delta| \) is established.
\end{proof}

\noindent The next lemma is needed for performance analysis.

\begin{lemma}\label{lem:max-geq-half} 
    $$\max\left\{f_L(\mathbf{\omega^*}), \Gamma_\delta\right\}\geq \frac{1}{2}\text{OPT}_{\text{IP}},$$
where $\text{OPT}_{\text{IP}}$ is the optimal objective value of LPRC.
\end{lemma}

\begin{proof}
    Define the set $\overline{\mathcal{A}}_\delta:
:= \left\{ \omega \in \mathcal{A} : \omega_{b,l} = 0 \text{ if } l \notin \mathcal{L}_b^\delta \right\}.$
Given the preceding definitions, for each \( \omega \in \mathcal{A} \), vectors \( \omega' \in \mathcal{A}_\delta \) and \( \omega'' \in \overline{\mathcal{A}}_\delta \) exist such that \( \omega' + \omega'' = \omega \). Consequently, the inequality \( \Gamma_\delta + f_L(\omega') \geq f_L(\omega) \) holds. This follows from the observation that a feasible solution for \( \text{LP}(\omega) \) can be decomposed into distinct solutions for \( \text{LP}_\delta \) and \( \text{LP}(\omega') \).
 Thus we have:
        $\max\left\{f_L(\mathbf{\omega^*}), \Gamma_\delta\right\}\geq \frac{1}{2}\max_{\omega\in \mathcal{A}_\delta}\left(\Gamma_\delta+f_L(\omega)\right)\geq \frac{1}{2}\max_{\omega\in \mathcal{A}}f_L(\omega)\geq \frac{1}{2}\text{OPT}_{\text{IP}}.$
\end{proof}

\noindent Combining Theorem \ref{thm:main}, Theorem \ref{thm:alg-small-costs} and Lemma \ref{lem:max-geq-half} immediately leads to the following result.

\begin{theorem}
    Given a constant $\eta'\in (0,1/4)$, the expected approximation ratio of the solution obtained by Algorithm \ref{alg:approx-all-costs} is at least $\frac{1}{2}-\frac{1}{2e}-\eta'$.
\end{theorem}

\bigskip

\section{LPRC with $\tau-$Tolerance in Resource Constraints}\label{sec:tolerance}

In this section, we present a randomized approximation algorithm with expected approximation ratio \(\left(1-\frac{1}{e}-\eta\right)\) while the solution utilizes no more than \(1+\tau\) units of each resource type, for any given constants \(\eta, \tau \in \left(0, \frac{1}{2}\right)\). This approximation ratio is \emph{nearly tight}. Specifically, when all costs are set to \(0\), the hardness result from Theorem \ref{thm:hardness} can be directly applied here.


\noindent {\bf Algorithm Overview:} In the algorithm, we enumerate all feasible line assignments restricted to the lines with at least one type of cost exceeding a predefined threshold; such lines will be called \emph{high-cost lines}, and such feasible assignments will be called \emph{high-cost assignments}. As shown in Section \ref{sec:approx-LPRC}, the number of high-cost assignments is polynomial. For each enumerated case of a high-cost assignment, we adjust the remaining resources by removing those consumed by the assigned lines, and then augment each type of the residual resource by  $\tau$. After scaling, the adjusted residual resources can then be used for the lines with all costs below the threshold (\emph{low-cost lines}). Then, a modified version of the LP relaxation (\ref{LP-relaxation}) is solved for each enumerated case where assignment of large-cost lines is fixed according to the case, and the case with the largest objective value of the modified LP relaxation is selected. Then, Algorithm  \ref{alg:low-costs-feasible} is applied to the given optimal solution of the modified LP relaxation for the selected case, and the obtained solution is presented as the final output.

\noindent Below, we define the modified LP relaxation. The notation from the previous section applies.

\begin{definition}
    Given \( \delta \leq \tau \in (0, \frac{1}{2}) \), for each feasible vector \( \omega \in \mathcal{A}_\delta \), let \( \text{LP}_{\delta,\tau}(\omega) \) denote the modified LP relaxation (\ref{LP-relaxation}). In \( \text{LP}_{\delta,\tau}(\omega) \), the following adjustments are made:
    \begin{enumerate}
        \item Add constraints \( \lambda_{b,l}^{(t)} = 0 \) whenever \( l \notin \mathcal{L}_b^\delta \) and $\omega_{b,l}=0$.
        \item Add constraints \( \lambda_{b,l}^{(t)} = 0 \) whenever \( \sum_{l'\in\mathcal{L}_b}\omega_{b,l'}= 1\) and $\omega_{b,l}=0$.
        \item Replace \( c_{b,l}^{(k)}\) with \( \widehat{c}_{b,l}^{(k)} \), where  \( \widehat{c}_{b,l}^{(k)} = 0\) if \(  l \notin \mathcal{L}_b^\delta  \), and \( \widehat{c}_{b,l}^{(k)} = \frac{{c}_{b,l}^{(k)}}{1 - \sum_{b,l} c_{b,l}^{(k)}\omega_{b,l}+ \tau}\) otherwise.
        \item Define \( f_L^{{\delta,\tau}}: \mathcal{A}_\delta \rightarrow \R_+ \). For each \( \omega \in \mathcal{A}_\delta \), let \( f_L^{\delta,\tau}(\mathbf{\omega}) \) be the optimal objective value of \( \text{LP}_{\delta,\tau}(\mathbf{\omega}) \).
    \end{enumerate}
    
\end{definition}

\noindent Step 1 filters out high-cost lines not in \( \omega \). Step 2 ensures buses assigned to lines in \( \omega \) cannot be reused. Step 3 sets costs to \( 0 \) for high-cost lines (\( \notin \mathcal{L}_b^\delta \)) and scales costs for low-cost lines to reflect resource usage and the adjusted limit \( \tau \), given all resource constraints in the LP relaxation (\ref{LP-relaxation}) have a right-hand side of 1.

Since we apply scaling of the costs of low-cost lines, we need to make sure that after the scaling the conditions of applicability of 
Algorithm \ref{alg:low-costs-feasible} are still met. This is done by a proper choice of $\delta$.

\begin{lemma}\label{lem:residual-small-cost}
    Given $\tau, \eta\in (0,\frac{1}{2})$, let  $q = \frac{4}{\eta}$ and \( \delta = \frac{1}{4q^3K^3}\cdot \tau \). For any vector \( \omega \in \mathcal{A}_\delta \) and its associated \( \text{LP}_{\delta,\tau}(\mathbf{\omega}) \), the parameter \( \left(\widehat{c}_{b,l}^{(k)}\right)_{b\in[M],l\in\mathcal{L}_b,k\in[K]} \) satisfies the criterion specified in Assumption \ref{assumption:small-costs}. Specifically, we have:
\[
\max_{b \in [M], l \in \mathcal{L}_b, k \in [K]} \widehat{c}_{b,l}^{(k)} \leq \frac{\eta^3}{256 K^3}.
\]

\noindent Therefore, we can apply Algorithm \ref{alg:low-costs-feasible} to $\text{LP}_{\delta,\tau}(\mathbf{\omega})$ with costs set as \( \left(\widehat{c}_{b,l}^{(k)}\right)_{b\in[M],l\in\mathcal{L}_b,k\in[K]} \).
\end{lemma}

\begin{proof}
    The result is established through the following inequalities:

$~~~~~~~~~~~~~~~~~~~~~~~~~~~~
    \widehat{c}_{b,l}^{(k)} \leq \frac{\delta}{1 - \sum_{b,l} c_{b,l}^{(k)}\omega_{b,l}+ \tau}  \leq \frac{\delta}{\tau}  \leq \frac{1}{4q^3 K^3}  = \frac{\eta^3}{256 K^3}.
$
\end{proof}

\noindent We are now ready to outline the algorithm, guided by Lemma \ref{lem:residual-small-cost} in choosing the line cost threshold based on \( \eta \) and \( \tau \).

\begin{namedalg}{C-Tol}\label{alg:tolerance}

~

   \noindent {\bf Input:}
        An instance of LPRC (\ref{eq:ILP-original}), and $\tau,\eta\in \left(0,\frac{1}{2}\right)$.

\noindent {\bf Output: }A (randomized) solution to (\ref{eq:ILP-original}).

\noindent {\bf Steps: }
\begin{enumerate}
    \item Define $q = \frac{4}{\eta}$ and \( \delta = \frac{1}{4q^3K^3}\cdot \tau \).
    \item  
     For each \( \omega \in \mathcal{A}_\delta \), compute \( f_L^{{\delta,\tau}}(\omega) \).
      Let $\omega^*:=\argmax_{\omega\in\mathcal{A}_\delta}f_L^{{\delta,\tau}}(\omega)$.
    
    \item Apply Algorithm \ref{alg:low-costs-feasible} to $\text{LP}_{\delta,\tau}(\omega^*)$ with costs set as \( \left(\widehat{c}_{b,l}^{(k)}\right)_{b\in[M],l\in\mathcal{L}_b,k\in[K]} \). To clarify, the final solution adheres to the following conditions:
\(
\sum_{b \in [M], l \in \mathcal{L}_b} \widehat{c}_{b,l}^{(k)} \cdot \widetilde{\lambda}_{b,l} \leq 1.
\)
Here, \(\left( \widetilde{\lambda}_{b,l} \right)_{b \in [M], l \in \mathcal{L}_b}\) is a binary vector that indicates whether bus $b$ operates on line $l$ in the final solution.\hfill

\end{enumerate}

\end{namedalg}

\noindent The following results provide theoretical performance guarantees for Algorithm \ref{alg:tolerance}.

\begin{theorem}\label{thm:C-Tol}
Algorithm \ref{alg:tolerance} has an expected approximation ratio of at least \( \left(1-\frac{1}{e}-\eta\right) \) while using at most \( 1+\tau \) units of each resource.
\end{theorem}

\subsection{Proof of Theorem \ref{thm:C-Tol}}\label{app-sec:thm:C-Tol}

We firstly need to prove the following Lemma, which indicates the lower bound of $\max_{\omega\in\mathcal{A}_\delta}f_L^{{\delta,\tau}}(\omega)$. 

\begin{lemma}\label{lem:max-geq-opt}
     Given an instance of LPRC (\ref{eq:ILP-original}) and $\tau,\eta\in (0,\frac{1}{2})$, let $\text{OPT}_{\text{IP}}$ be the optimal objective value for the instance. With the same notations as in Algorithm \ref{alg:tolerance}, we have that $\max_{\omega\in\mathcal{A}_\delta}f_L^{{\delta,\tau}}(\omega)\geq \text{OPT}_{\text{IP}}.$
\end{lemma}

\begin{proof}
    For an optimal solution to (\ref{eq:ILP-original}), let $\widehat{\lambda}_{b,l} = 1$ if bus $b$ operates on line $l$, and $0$ otherwise. Let $\widehat{\xi}_{b,l,(i,j)}$ be the number of passengers from OD pair $(i,j)$ assigned to bus $b$ on line $l$ in the solution. Let $\widetilde{\omega}_{b,l} = \widehat{\lambda}_{b,l}$ if $l\notin \mathcal{L}_b^\delta$, and $0$ otherwise. Then according to the definition of $\mathcal{A}_\delta$, we have that $\mathbf{\widetilde{\omega}}:=\left(\widetilde{\omega}_{b,l}\right)_{b\in [M], l\in \mathcal{L}_b}\in \mathcal{A}_\delta$. Without loss of generality, assume $\theta_{b,l,(i,j)}^{(1)} = \widehat{\xi}_{b,l,(i,j)}$,  $\lambda_{b,l}^{(1)} = \widehat{\lambda}_{b,l}$ and $\lambda_{b,l}^{(t)} = 0$ for all $t\geq 2$, $(i,j)\in D$, $b\in [M]$ and $l\in \mathcal{L}_b$. Then it can be checked that this defines a feasible solution to $\text{LP}_{\delta,\tau}\left(\mathbf{\widetilde{\omega}}\right)$, and its objective value is $\text{OPT}_{\text{IP}}$. Therefore, we have that $\text{OPT}_{\text{IP}}\leq f_L^{{\delta,\tau}}\left(\mathbf{\widetilde{\omega}}\right)\leq \max_{\omega\in\mathcal{A}_\delta}f_L^{{\delta,\tau}}(\omega)$.
\end{proof}

\begin{proof}[Proof of Theorem \ref{thm:C-Tol}]
Due to Lemma~\ref{lem:residual-small-cost}, we can directly apply Theorem~\ref{thm:alg-small-costs} to demonstrate that the expected objective value of the solution returned by Algorithm~\ref{alg:tolerance} is at least \(\left(1-\frac{1}{e}-\eta\right)f_L^{{\delta,\tau}}(\omega^*)\). This, in conjunction with Lemma~\ref{lem:max-geq-opt}, establishes the approximation ratio.

Let $\left(\widetilde{\lambda}_{b,l}\right)_{b\in [M], l\in \mathcal{L}_b}$ be the binary vector that indicates whether bus $b$ operates on line $l$ in the solution returned by Algorithm \ref{alg:tolerance}. 
By the definition of \(\text{LP}_{\delta,\tau}(\mathbf{\omega^*})\) and Algorithm~\ref{alg:low-costs-feasible}, we observe that \(\widetilde{\lambda}_{b,l} = 0\) whenever \(l\notin \mathcal{L}_b^\delta\) and $\omega_{b,l}^*=0$, and \(\widetilde{\lambda}_{b,l} \leq 1\) otherwise. This leads to
\begin{align*}
    \sum_{b\in [M], l\in \mathcal{L}_b}c_{b,l}^{(k)}\cdot \widetilde{\lambda}_{b,l}\leq&  \sum_{b,l} c_{b,l}^{(k)}\omega^*_{b,l} + \sum_{b\in [M], l\in \mathcal{L}_b^\delta} c_{b,l}^{(k)}\cdot \widetilde{\lambda}_{b,l}\\
    =&  \sum_{b,l} c_{b,l}^{(k)}\omega^*_{b,l} + \left(1 - \sum_{b,l} c_{b,l}^{(k)}\omega^*_{b,l}+ \tau\right)\cdot \sum_{b\in [M], l\in \mathcal{L}_b^\delta} \frac{c_{b,l}^{(k)}}{1 - \sum_{b,l} c_{b,l}^{(k)}\omega^*_{b,l}+ \tau}\cdot \widetilde{\lambda}_{b,l}\\
    =&  \sum_{b,l} c_{b,l}^{(k)}\omega^*_{b,l} + \left(1 - \sum_{b,l} c_{b,l}^{(k)}\omega^*_{b,l}+ \tau\right)\cdot \sum_{b\in [M], l\in \mathcal{L}_b^\delta} \widehat{c}_{b,l}^{(k)}\cdot \widetilde{\lambda}_{b,l}\leq 1+\tau
\end{align*}
Hence, the upper bound on the utilization of the \(k\)-th resource is \(1 + \tau\).
\end{proof}

\bigskip

\section{Proof of Proposition \ref{prop:alg-R-approx}}\label{app-sec:prop:alg-R-approx-proof}

\begin{proof}
The proof uses an analysis similar to that in Section \ref{sec:low-cost}. It is important to note that the approximation guarantee here is derived only from the randomized solutions obtained by Step \ref{alg-step:LC-M-rounding} of Algorithm \ref{alg:low-cost-modified} that meet all resource constraints, ignoring the rewards of generated infeasible solutions (as for Algorithm \ref{alg:low-costs-feasible} in Appendix \ref{sec:low-cost}). Subsequent steps of Algorithm \ref{alg:low-cost-modified} can only improve upon this guarantee.

Let $Y_{b,l}$ be the random indicator variable denoting the event that bus $b$ is assigned to line $l$ in Step \ref{alg-step:LC-M-sample}. Note that if $Y_{b,l}=0$ for all $l\in \mathcal{L}_b$, then it means that bus $b$ is not utilized in the randomized solution. For a given $k\in [K]$ and $\rho \geq 1$, we have that 
    \begin{align}
        &P\left(\sum_{b\in [M]}\sum_{l\in \mathcal{L}_b}Y_{b,l}\cdot c_{b,l}^{(k)}\geq (1+\rho)(1-\epsilon)\right)
        =P\left(\sum_{b\in [M]}\sum_{l\in \mathcal{L}_b}Y_{b,l}\cdot c_{b,l}^{(k)}\cdot J\geq (1+\rho)\cdot J\cdot (1-\epsilon)\right)\nonumber\\
        \leq& \left(\frac{e^{\rho}}{(1 + \rho)^{1 + \rho}}\right)^{J \cdot (1 - \epsilon)}
        \leq 
        \left\{
\begin{array}{ll}
     e^{-\rho\cdot J\cdot  (1-\epsilon)/3} & \text{if } \rho \geq 1, \\
     e^{-\rho^2\cdot J\cdot  (1-\epsilon)/3} & \text{if } \rho \leq 1.
\end{array}
\right.\label{eq:chernoff}
    \end{align}
    The last two inequalities are due to  $\sum_{l\in \mathcal{L}_b}Y_{b,l}\cdot c_{b,l}^{(k)}\cdot J\in [0,1]$, and $\sum_{b\in [M]}\sum_{l\in \mathcal{L}_b}\E \left[Y_{b,l}\right]\cdot c_{b,l}^{(k)}\leq  (1-\epsilon)$,  and the regular derivations of Chernoff bounds. 

For a non-negative integer $\psi$, let \( \mathcal{B}_\psi \) be the event that  Step \ref{alg-step:LC-M-rounding}  returns a randomized solution which utilizes at least \( \frac{\psi}{J} \) units for one or more resources, but does not exceed \( \frac{\psi+1}{J} \) units for any resource. We define $P_\psi := P(\mathcal{B}_\psi)$ and
\[
E_\psi := \E\left[\text{Objective value of the randomized solution generated by  Step \ref{alg-step:LC-M-rounding}} ~\Big\vert~ \mathcal{B}_\psi\right].\]


Let $\Gamma$ be the objective value of the optimal LP solution, then based on Theorem \ref{thm:main}, we have that
\begin{align*}
    &\left(1-\frac{1}{e}\right)\cdot \left(1-\epsilon\right)\Gamma
    \leq  \sum_{i=0}^{J-1}E_{i}\cdot P_{i} + \sum_{i=J}^{\infty}E_{i}\cdot P_{i} 
    \leq \sum_{i=0}^{J-1}E_{i}\cdot P_{i} + \sum_{i= J}^{\infty}\frac{i+1}{J}\cdot \Gamma\cdot P_{i}\\
    =& \sum_{i=0}^{J-1}E_{i}\cdot P_{i} + \sum_{i= J}^{\infty}\left(1+\frac{i-J+1}{J}\right)\cdot \Gamma\cdot P_{i}
    =\sum_{i=0}^{J-1}E_{i}\cdot P_{i} + \Gamma\cdot \sum_{i = J}^\infty P_{i} + \frac{\Gamma}{J}\sum_{i= J}^{\infty}\sum_{j = i}^\infty P_{j}
\end{align*}
Thus, we have 

\begin{align}\label{eq:LC-M-approx-LB}
    \sum_{i=0}^{J-1}E_{i}\cdot P_{i} \geq \Gamma\cdot\left[ \left(1-\frac{1}{e}\right)\cdot \left(1-\epsilon\right) -  \sum_{j = J}^\infty P_{j} - \frac{1}{J}\sum_{i= J}^{\infty}\sum_{j = i}^\infty P_{j}\right].
\end{align}

\noindent By the definition of $P_\psi$, we have that

$$\sum_{j = i}^\infty P_{j} = P( \mbox{The randomized solution
utilizes at least $\frac{i}{J}$ units for one or more resources.} )$$

\noindent Recall that  $\delta := \frac{\epsilon}{1-\epsilon}$, and $\delta_i := \delta +\frac{i-J}{J\cdot (1-\epsilon)}$. Then by (\ref{eq:chernoff}) we have that

\begin{align}\label{modified-alg-ineq-1}
    \sum_{j = J}^\infty P_{j}
\leq K\cdot \left(\frac{e^{\delta}}{(1 + \delta)^{1 + \delta}}\right)^{J \cdot (1 - \epsilon)}
\end{align}


\begin{align*}
    &\sum_{i= J}^{\infty}\sum_{j = i}^\infty P_{j}
\leq K\cdot \left(\sum_{i= J}^{\widehat{J}-1} \left(\frac{e^{\delta_i}}{(1 + \delta_i)^{1 + \delta_i}}\right)^{J \cdot (1 - \epsilon)} 
+ \sum_{i= \widehat{J}}^{\infty}e^{-\delta_i\cdot J\cdot  (1-\epsilon)/3}\right).
\end{align*}

\noindent Since

\begin{align*}
\sum_{i= \widehat{J}}^{\infty}e^{-\delta_i\cdot J\cdot  (1-\epsilon)/3} =& \sum_{i= \widehat{J}}^{\infty}e^{-\left(\delta +\frac{i-J}{J\cdot (1-\epsilon)}\right)\cdot J\cdot  (1-\epsilon)/3} 
= e^{-\epsilon\cdot {J} /3} \sum_{i= \widehat{J}-J}^{\infty}e^{-i/3}\\
\leq &  e^{-\epsilon\cdot {J} /3} \cdot \left( e^{- \left(\widehat{J} - J\right)/3} + \int_{x = \widehat{J}-J}^\infty e^{-x/3}dx\right)
=  e^{-\epsilon\cdot {J} /3} \cdot \left(4\cdot e^{- \left(\widehat{J} - J\right)/3}\right),
\end{align*}

\noindent so we have 
\begin{align}\label{modified-alg-ineq-2}
    \sum_{i= J}^{\infty}\sum_{j = i}^\infty P_{j} \leq K\cdot \left(\sum_{i= J}^{\widehat{J}-1} \left(\frac{e^{\delta_i}}{(1 + \delta_i)^{1 + \delta_i}}\right)^{J \cdot (1 - \epsilon)}  + 4e^{-(\epsilon\cdot {J} + \widehat{J} - J) /3}\right)
\end{align}

\noindent By substituting (\ref{modified-alg-ineq-1}) and (\ref{modified-alg-ineq-2}) into (\ref{eq:LC-M-approx-LB}), we directly obtain the desired result.
\end{proof}

\begin{remark}
    Although the analysis in Section \ref{sec:low-cost} employs weaker bounds compared to the analysis presented in this section, it is sufficient to derive the desired result and ensures a clear and understandable proof.
\end{remark}

\section{Two heuristics involved in Algorithm \ref{alg:low-cost-modified}}\label{sec:heuristics}

\noindent \textbf{Heuristic H1}.\\
A solution that does not meet some resource constraints is given. For each resource type $k$, let $C_k$ be the solution's total cost for resource type $k$ if this resource is overused, and be $1$ otherwise. For each used bus $b$ assigned to line $l$, its excess cost for each resource type $k$ is calculated as $\min \left \{c^{(k)}_{b,l}, C_k-1 \right\}$.  Then, the buses used by the solution are ranked by their reward-cost ratio, which is calculated as the ratio of the reward generated by the bus to the sum of the bus' excess costs across all resource types. The bus with the lowest reward-cost ratio is removed, then the ranking is updated, and the process is repeated until all resource constraints are met.

\noindent \textbf{Heuristic H2}.\\
\emph{This heuristic assigns the remaining unserved passengers to the \emph{already-selected} buses only (no new buses are added), using residual arc capacities along their selected lines.}

\begin{enumerate}
\item Compute residual demands $d_{(i,j)}$ and residual capacities $r_{b,l}(e) := C_b - \sum_{(i,j) \in D^{b,l}(e)} \xi_{b,l,(i,j)}$ for each selected bus-line pair $(b,l)$ and arc $e \in l$, where $\xi_{b,l,(i,j)}$ represents passengers already assigned from the temporary solution after Step~\ref{alg-step:LC-M-remove} (i.e., after Heuristic H1).

\item  For each selected bus-line pair $(b,l)$ and OD pair $(i,j) \in D^{b,l}$ with $\tilde{d}_{(i,j)} > 0$, compute the bottleneck capacity $\mathrm{bn}_{b,l,(i,j)} := \min_{e \in E_{(i,j)}} r_{b,l}(e)$ where $E_{b,l,(i,j)}$ is the set of arcs traversed. Include $(b,l,(i,j))$ in candidate set $\mathcal{T}$ if $\mathrm{bn}_{b,l,(i,j)} > 0$.

\item Sort $\mathcal{T}$ by: (1) nonincreasing value $v_{b,l,(i,j)}$, (2) nondecreasing subpath length $|E_{b,l,(i,j)}|$, (3) nonincreasing bottleneck $\mathrm{bn}_{b,l,(i,j)}$.

\item For each $(b,l,(i,j)) \in \mathcal{T}$ in sorted order:
\begin{itemize}
\item Skip if $\tilde{d}_{(i,j)} = 0$.
\item Recompute current bottleneck: $\mathrm{bn} := \min_{e \in E_{b,l,(i,j)}} r_{b,l}(e)$.
\item Set $h := \min\{\tilde{d}_{(i,j)}, \mathrm{bn}\}$.
\item If $h > 0$, update: $\xi_{b,l,(i,j)} \leftarrow \xi_{b,l,(i,j)} + h$, $\tilde{d}_{(i,j)} \leftarrow \tilde{d}_{(i,j)} - h$, and $r_{b,l}(e) \leftarrow r_{b,l}(e) - h$ for all $e \in E_{b,l,(i,j)}$.
\end{itemize}

\item The heuristic completes after examining each candidate in $\mathcal{T}$ exactly once. The algorithm makes a single pass through the sorted candidate list without iteration. Some candidates may result in no assignment if capacity becomes unavailable due to earlier assignments. Any remaining unserved demand after all candidates are processed remains unassigned.
\end{enumerate}

\end{APPENDICES}

\end{document}